\documentclass[preprint,12pt,3p]{elsarticle}

\usepackage{amsmath,bm}
\usepackage{amssymb}
\usepackage{amsthm}
\usepackage{mathtools}

\theoremstyle{plain}
\newtheorem{theorem}{Theorem}[section]
\newtheorem{lemma}[theorem]{Lemma}
\newtheorem{cor}[theorem]{Corollary}
\newtheorem{prop}[theorem]{Proposition}
\theoremstyle{definition}
\newtheorem{defn}[theorem]{Definition}

\newtheorem{assumption}[theorem]{Assumption}
\newtheorem{example}[theorem]{Example}
\theoremstyle{remark}
\newtheorem{rem}[theorem]{Remark}

\makeatletter
\def\rom#1{\mbox{\leavevmode\skip@\lastskip\unskip\/%
           \ifdim\skip@=\z@\else\hskip\skip@\fi{\rm{#1}}}}
\makeatother

\makeatletter
\def\ps@pprintTitle{%
     \let\@oddhead\@empty
     \let\@evenhead\@empty
     \def\@oddfoot{%
       \hfill\thepage \hfill}%
     \let\@evenfoot\@oddfoot}
\makeatother
\newcommand{\D}{\mathbb{D}}
\newcommand{\E}{\mathbb{E}}
\newcommand{\N}{\mathbb{N}}
\newcommand{\R}{\mathbb{R}}
\newcommand{\cB}{\mathcal{B}}
\newcommand{\cE}{\mathcal{E}}
\newcommand{\cL}{\mathcal{L}}
\newcommand{\cT}{\mathcal{T}}
\renewcommand{\a}{\alpha}\renewcommand{\b}{\beta}
\newcommand{\gm}{\gamma}\newcommand{\dl}{\delta}
\newcommand{\eps}{\varepsilon}\newcommand{\zt}{\zeta}
\newcommand{\lm}{\lambda}\newcommand{\kp}{\kappa}
\newcommand{\sg}{\sigma}
\newcommand{\ph}{\phi}
\newcommand{\om}{\omega}
\newcommand{\Gm}{\Gamma}
\newcommand{\Om}{\Omega}
\newcommand{\bone}{\mathbf{1}}
\newcommand{\la}{\langle}
\newcommand{\ra}{\rangle}
\newcommand{\wg}{\wedge}
\newcommand{\sd}{\mathsf{d}}
\newcommand{\nb}{D}
\newcommand{\del}{\partial}
\newcommand{\loc}{{\mathrm{loc}}}
\newcommand{\qad}{\phantom{\le{}}}
\DeclareMathOperator{\Dom}{Dom}
\newcommand{\relmiddle}[1]{\mathrel{}\middle#1\mathrel{}}
\DeclareMathOperator*{\essinf}{ess\,inf}
\DeclareMathOperator*{\esssup}{ess\,sup}
\numberwithin{equation}{section}
\allowdisplaybreaks[3]
\begin{document}

\begin{frontmatter}



\title{%
An integrated version of Varadhan's asymptotics for lower-order
perturbations of strong local Dirichlet forms\tnoteref{t1}}
\tnotetext[t1]{This study was supported by  JSPS KAKENHI Grant Number JP15H03625.}

\author{Masanori Hino}
\ead{hino@math.kyoto-u.ac.jp}
\address{Department of Mathematics, Kyoto University, Kyoto 606--8502, Japan}
\author{Kouhei Matsuura}
\ead{kouhei.matsuura.r3@dc.tohoku.ac.jp}
\address{Mathematical Institute, Tohoku University, Aoba, Sendai 980--8578, Japan}
\begin{abstract}
The studies of Ram\'irez~\cite{R}, Hino--Ram\'irez~\cite{HR}, and Ariyoshi--Hino~\cite{AH} showed that
an integrated version of Varadhan's asymptotics 
holds for Markovian semigroups associated with arbitrary strong local symmetric Dirichlet forms.
In this paper, we consider 
non-symmetric
bilinear forms that are the sum of strong local symmetric Dirichlet forms and lower-order perturbed terms. We give sufficient conditions for the associated semigroups to have asymptotics of the same type. 
\end{abstract}


\end{frontmatter}

\tableofcontents
\section{Introduction}
Let $(E,\cB,\mu)$ be a $\sg$-finite measure space and $(\cE^{0},\D)$ a symmetric strong local Dirichlet form on the $L^2$ space of $(E,\cB,\mu)$.
 Let $\{T_{t}^{0}\}_{t>0}$ denote the semigroup associated with $(\cE^{0},\D)$, and set
 $P_{t}^{0}(A,B)=\int_{A} T_{t}^{0} \bone_{B}\,d\mu$ for $t>0$ and $A,B \in \cB$ with positive and finite measure. 
  In \cite{AH}, the following small-time asymptotic estimate for $\{T_{t}^{0}\}_{t>0}$ was proved as a generalization of results from previous work~\cite{R,H2,HR}:
\begin{equation}
 \lim_{t \to 0} t \log P_{t}^{0}(A,B)=-\frac{\sd(A,B)^{2}}{2}. \label{eq:varadhan}
\end{equation}
  Here, $\sd(A,B)$ is the intrinsic distance between $A$ and $B$, which can be determined from only $(\cE^{0},\D)$ (see {\cite[p.\ 1241]{AH}} or Definition~\ref{def:intrinsic} below for details). 
Similar small-time asymptotics of transition densities have been studied extensively. These are usually called Varadhan-type estimates,
in reference to~\cite{V}. 
In particular, that the estimate holds was proved in \cite{N} for a class of symmetric and uniform elliptic diffusion processes on Lipschitz manifolds. This is one of the most general results.
Asymptotics of the form~\eqref{eq:varadhan} can be considered as an 
integrated version of Varadhan's asymptotics.

  The purpose of this paper is to extend the formula \eqref{eq:varadhan} to a class of non-symmetric bilinear forms.   
  Specifically, we first assume that $(\cE^0,\D)$ mentioned above is expressed as
\[
\cE^0(f,g)=\frac12\int_E (\nb f,\nb g)_H\,d\mu,\quad f,g\in\D,
\]
 where $\nb$ is a first-order derivation operator taking values in a separable Hilbert space $H$. 
Our main object is to obtain small-time asymptotics for a non-symmetric form $(\cE,\D)$ given by the sum of $\cE^{0}$ and the lower-order perturbed term $\int_{E}(b,D f)_{H}g \,d\mu+\int_{E}(c,D g)_{H}f \,d\mu+\int_E V  fg\,d\mu$ (see \eqref{eq:cE}).
 When the perturbed term is small relative to $\cE^0$, the form $(\cE,\D)$ becomes a lower-bounded bilinear form and has an associated positivity-preserving semigroup $\{T_t\}_{t>0}$ on $L^2(E,\mu)$.
 For measurable sets $A$ and $B$ having positive and finite $\mu$-measure, let $P_t(A,B)=\int_A T_t\bone_B\,d\mu$, as before.
  We study the conditions on $b$, $c$, and $V$ that suffice for the semigroup $\{T_t\}_{t>0}$ to have the same integrated Varadhan's asymptotics as $\{T_{t}^{0}\}_{t>0}$. That is, for
  \begin{equation}\label{eq:asymp}
\lim_{t\to0}t \log P_{t}(A,B)=\lim_{t\to0}t \log P_{t}^{0}(A,B)=-\frac{\sd(A,B)^{2}}{2}.
   \end{equation}   
What kind of restrictions should we impose on $b$, $c$, and $V$ to guarantee the validity of \eqref{eq:asymp}? 
It is reasonable to expect that \eqref{eq:asymp} would hold if they were sufficiently smaller than $\cE^0$ 
in terms
of quadratic forms.
From another perspective, 
we can make a probabilistic argument, exemplified in the following typical case.
  Let $(E,H,\mu)$ be an abstract Wiener space, and suppose that $(\cE^{0},\D)$ and $(\cE,\D)$ are defined as  \begin{align*}
  \cE^{0}(f,g)&=\frac{1}{2}\int_{E}(\nb f,\nb g)_H\,d\mu,\\
  \cE(f,g)&=\cE^{0}(f,g)+\int_{E}(b, \nb f)_{H}g\,d\mu,
  \quad f,g\in \D:=\D^{1,2},
  \end{align*}
  where $D$ denotes the $H$-derivative in the Malliavin calculus, $b$ 
is
an $H$-valued measurable function on $E$, and $\D^{1,2}$ is the first-order $L^2$-Sobolev space on $E$. 
If $\exp(\gm |b|^{2}_{H})$ is $\mu$-integrable for some $\gm>8$, then by using the logarithmic Sobolev inequality, we can prove that $(\cE,\D)$ is 
well-defined as
a lower-bounded bilinear form and that there exists a corresponding semigroup $\{T_t\}_{t>0}$ on $L^2(E,\mu)$ (see Example~\ref{ex:2}).
Moreover,  
$\{T_t\}_{t>0}$ has 
a probabilistic representation as
\[
T_t f(x)=\E_x\left[ f(X_t)\exp\left(M_t-\frac12\la M\ra_t\right)\right],
\]
where $(\{X_t\}_{t\ge0},\{\mathbb{P}_x\}_{x\in E})$ is the Ornstein--Uhlenbeck process associated with $(\cE^0,\D)$ and $\{M_t\}_{t\ge0}$ is a martingale additive functional suitably associated with $b$ (see, e.g., \cite{H}). Note, in particular, 
that the quadratic variation of $M$ is given by
$\la M\ra_t=\int_0^t |b(X_s)|_H^2\,ds$.
From H\"{o}lder's inequality, for measurable sets $A$ and $B$ with positive 
$\mu$-measure,
\begin{align}
\int_A T_{t}\bone_{B}\,d\mu 
&\le \left(\int_A \E_{x}\left[\bone_{B}(X_{t})\right]d\mu\right)^{1/p}  \left(\int_{A}\E_{x} \left[\exp \left( 2q M_t-\frac12\la 2qM\ra_t\right) \right]d\mu \right)^{1/2q} \notag\\
&\qad \times \left(\int_{A}\E_{x} \left[\exp\left((2q^{2}-q)\la M\ra_t \right) \right]d\mu \right)^{1/2q},
\label{eq:intro}
\end{align}
where $p>1$ and $q$ is the 
conjugate exponent of $p$. 
The first term of the right-hand side is $P^0_t(A,B)^{1/p}$.
The second term is dominated by $\mu(A)^{1/2q}$. 
The third term is estimated by Jensen's inequality:
\begin{equation*}
\int_{A}\E_{x} \left[\exp\left((2q^{2}-q)\int_{0}^{t}\left|b(X_{s})\right|_H^{2}ds \right) \right]d\mu \le \frac{1}{t}\int_{0}^{t}\int_{A}\E_x\left[\exp \left((2q^{2}-q)t\left|b(X_s)\right|_H^{2}\right)\right]d\mu\,ds.
\end{equation*}
By the exponential integrability of $|b|_H^2$, the right-hand side is equal to 
\[
\frac1t\int_0^t \int_A T^0_s\left(\exp\left((2q^2-q)t|b|_H^2\right)\right)d\mu\,ds
=\frac1t\int_0^t \int_E (T^0_s\bone_A)\exp\left((2q^2-q)t|b|_H^2\right)d\mu\,ds,
\]
which is finite for sufficiently small positive values of $t$ and converges to $\mu(A)$ as $t\to0$.
Combining 
these estimates
and the asymptotics with respect to 
$(\cE^0,\D)$ 
and letting $p\to1$, we obtain the upper estimate
\begin{equation*}
\varlimsup_{t \to 0} t\log P_{t}(A,B) \le -\frac{\sd(A,B)^{2}}{2}.
\end{equation*}
This kind of probabilistic argument is applicable to more general situations, by using a generalized Cameron--Martin--Maruyama--Girsanov formula (see, e.g., \cite{LLZ,RZ,FK}). 
The exponential integrability condition imposed above is not exactly consistent with smallness in the sense of quadratic forms. 
Indeed, 
in the estimate of \eqref{eq:intro}, we used
the fact that $\exp(\gm|b|_H^2)$ is $\mu$-integrable for only some $\gm>0$.
Therefore, it is reasonable to consider two types of smallness---smallness in term of quadratic forms and in terms of some exponential integrability---in describing the conditions sufficient for \eqref{eq:asymp}.

In this paper, we introduce conditions that take the observation above into consideration 
(see conditions (B.2)$_{A,B}$, (B.2$'$), and Proposition~\ref{prop:B2})
and prove the upper estimate under their assumptions (Theorem~\ref{th:upper}). 
Moreover, we prove that the lower estimate holds under minimal assumptions on $b$, $c$, and $V$ along with the assumption of 
the validity of
the upper estimate (Theorem~\ref{th:lower}). Combining these two results gives sufficient conditions for the integrated Varadhan estimates.
As in the previous studies~\cite{R,HR,AH}, the proof is purely analytic and only a measurable structure is imposed on the state space.
We remark that even for $b=c$, that is, even with $(\cE,\D)$ as a symmetric form, our results are new.

Because the proof is long, we briefly explain the broad ideas of the proof here. 
The upper estimate (Theorem~\ref{th:upper}) is proved in the spirit of Davies--Gaffney's method. 
In previous works~\cite{R,HR,AH}, they define $\sg(t)=\int_E (e^{\a w}T_t\bone_B)^2\,d\mu$ for given $\a>0$ and 
$w$ with $|\nb w|_H\le1$ $\mu$-a.e., and deduce the key differential inequality $\sg'(t)\le \a^2 \sg(t)$. 
Solving this inequality and optimizing it with respect to $\a$ and $w$ yields the desired estimate. 
Under the assumptions of our theorem, however, the perturbed terms cannot be controlled.
Instead, we define $\sg$ in the form $\sg(t)=\int_E (e^{\a w}T_t\bone_B)^{p(t)}\,d\mu$, where $p(t)=q-St$ with $q>2$ and $S>0$ being chosen suitably. Since $\{T_t\}_{t>0}$ can be extended to a semigroup on $L^p(\mu)$ for $p$ near $2$ in our setting, $\sg(t)$ is finite for small $t$. 
The variable exponent $p(t)$ means that the derivative of $\sg$ involves an extra logarithmic term, which suppresses the influence on $b$ and $c$. 
The price to pay for this is that the resulting differential inequality is 
coarser, in the form $\sg'(t)\le (1+\eps)\a^2\sg(t)+C\sg(t)\max\{0,-\log\sg(t)\}$.
Fortunately, the extra logarithmic term has no influence on the Varadhan-type estimate.
Introducing a variable exponent is a standard technique for estimating heat kernel densities (see, e.g., \cite{D}), but (unlike in such a context) $p(t)$ is taken to be a decreasing function in this study.
The definition of $\sg$ shown above is valid when $\mu$ is a finite measure; in general cases, we further need to modify the definition of $\sg$ (see \eqref{eq:sigma} and \eqref{eq:tu}) to avoid some technical obstacles. 
For this reason, we need a series of quantitative estimates, which makes the proof long.

The proof of the lower estimate is based on previous studies~\cite{R,HR,AH}, but the argument is more complicated due to the perturbed terms and the fact that the semigroup $\{T_t\}_{t>0}$ preserves positivity but is not Markovian.
We will outline the proof by the following formal argument.
We see the function $u_{t}=- t \log T_{t}\bone_{B}$ satisfies the relation
\begin{align*}
t\left(\partial_{t}u_{t}-\frac{\cL^{0}}{2}u_{t}\right)=u_{t}-\frac{1}{2}|D u_{t}|_{H}^{2}+(\text{extra terms involving $b$, $c$, and $V$}),
\end{align*}
where $\cL^{0}$ is the generator of $\{T^{0}_t\}_{t>0}$. 
(This identity corresponds to \eqref{identitypde} in the actual argument.)
If we assume for argument that the left-hand side converges to $0$ as $t\to0$ and the last term of the right-hand side is negligible, then the limit $u_0$ of $u_t$ (if it exists) will satisfy $|\nb u_0|_H^2=2u_0$.
What is actually obtained is an inequality of the form $|\nb u_0|_H^2\le 2u_0$, which implies
$|\nb\sqrt{2u_{0}}|_H \le 1$.
Furthermore, $u_{0}=0$ $\mu$-a.e.\ on $B$ should be satisfied.
Under these conditions, we have 
the formal inequality
\[
\lim_{t \to 0} \sqrt{-2t \log T_{t}\bone_{B}(x)} \le \sd(\{x\},B)
\]
by the definition of $\sd$, which is close to the lower-side estimate. 
Several difficulties, such as that $u_t$ is not necessarily bounded, make it hard to justify this procedure directly.
  To cope with problems such as the integrability (or not) of various terms and the existence (or not) of limits, 
we introduce a nice truncating function $\phi$ and bump functions $\{\chi_{k}\}$, 
and consider $\bar{\phi}_{t} \chi_{k}$ in place of  $u_{t}$, where  $\bar{\phi}_{t}=t^{-1}\int_{0}^{t}\phi(- s \log T_{s}\bone_{B})\,ds$. 
 Note that these cut-off functions are slightly different from those in \cite{AH,HR} in order to deal with the lack of the Markov property of $\{T_t\}_{t>0}$. 
This modification results in an increasing number of terms in the quantitative estimates as the proof progresses, which makes the proof longer and more technical than without the modification.
  
  This paper organized as follows. In Section~2, we introduce a framework and state the main theorems. 
Section~3 provides the proof of the upper estimate (Theorem~\ref{th:upper}). 
Section~4 provides the proof of the lower estimate (Theorem~\ref{th:lower}). 
In the last section, we prove some auxiliary propositions, discuss the conditions imposed on the theorems, and show some typical examples.

\section{Framework}
Let $(E,\cB,\mu)$ be a $\sg$-finite measure space, and
$H$ a  real separable Hilbert space. The inner product and norm of $H$ will be denoted by $(\cdot,\cdot)_H$ and $|\cdot|_H$, respectively.
The set of all real-valued measurable functions on $E$ is denoted by $L^0(\mu)$, where two functions are identified if they coincide on $\mu$-a.e.
For $p\in[1,\infty]$, the real $L^p$ space on $(E,\cB,\mu)$ is denoted by $L^p(\mu)$, and its norm by $\|\cdot\|_p$. The $L^2$ space of $H$-valued 
measurable 
functions on $(E,\cB,\mu)$ is  denoted by $L^2(\mu;H)$, and its norm by $\|\cdot\|_2$.

Let $\D$ be a dense subspace of $L^2(\mu)$, and $\nb$ be a closed linear operator from $L^2(\mu)$ to $L^2(\mu;H)$ with domain $\D$. 
We assume that $\nb$ has the following 
derivation property:
For arbitrary functions $f_1,f_2,\dots,f_m\in\D$ and $C^1$ functions $F$ on $\R^m$ with bounded first-order derivatives and $F(0)=0$, 
$F(f_1,\dots,f_m)$ belongs to $\D$ and
\begin{equation}\label{eq:chain}
\nb\bigl(F(f_1,\dots,f_m)\bigr)
=\sum_{j=1}^m \frac{\del F}{\del x_j}(f_1,\dots,f_m) \nb F_j.
\end{equation}
Then, a bilinear form $(\cE^0,\D)$ on $L^2(\mu)$, defined by
\[
\cE^0(f,g)=\frac12\int_E (\nb f,\nb g)_H\,d\mu,\quad f,g\in\D,
\]
is a Dirichlet form on $L^2(\mu)$. 
Moreover, this bilinear form has a strong local property: For any $f\in\D$ and $C^1$-functions $F,G$ on $\R$ with bounded first-order derivatives such that the supports of $F$ and $G$ are disjoint, $\cE(F(f)-F(0),G(f)-G(0))=0$. 
For other equivalent statements, see \cite[Proposition~I.5.1.3]{BH}, where this property is called a local property.

For $l\ge0$, we write $\cE^0_l(f,g)$ for $\cE^0(f,g)+l\int_E fg\,d\mu$. We also use $\cE^0(f)$ and $\cE^0_l(f)$ to denote  $\cE^0(f,f)$ and $\cE^0_l(f,f)$, respectively.
The space $\D$ becomes a Hilbert space with the inner product $(f,g)\mapsto \cE^0_1(f,g)$.
The following proposition is fundamental.
\begin{prop}\label{prop:D}
If a function $f\in\D$ is constant $\mu$-a.e.\ on a set $A\in\cB$, then $\nb f=0$ $\mu$-a.e.\ on $A$.
\end{prop}
\begin{proof}
Suppose that $f=\a$ $\mu$-a.e.\ on $A$ for some constant $\a$. 
Then, if $\a=0$, the conclusion follows from \cite[Proposition~I.7.1.4]{BH}. If $\a\ne0$, 
we can take a $C^1$ function $F$ on $\R$ with bounded derivative such that 
$F(0)=F(\a)=0$ and $F'(\a)\ne0$. Then, since $F(f)=0 $ $\mu$-a.e.\ on $A$,
\[
0=\nb(F(f))=F'(f)\nb f=F'(\a)\nb f\quad \text{$\mu$-a.e.\ on }A.
\]
This implies that $\nb f=0$ $\mu$-a.e.\ on $A$.
\end{proof}
For $A \in \cB$, we set
\begin{align*}
\D_A&=\{f\in\D\mid f=0\ \mu\text{-a.e.\ on } E\setminus A\},\\
\D_{A,b}&=\D_A\cap L^\infty(\mu), \\
\D_{A,b,+}&=\{ f \in \D_{A,b} \mid f \ge0 \ \mu\text{-a.e.} \}.
\end{align*}
We follow \cite{AH} in introducing the concept of measurable nests and related function spaces. 
\begin{defn}
An increasing sequence $\{E_k\}_{k=1}^\infty$ in $\cB$ is called a \emph{measurable nest}\footnote{In \cite{AH}, it is called just a nest.} if the following conditions are satisfied.
\begin{enumerate}
\item For every $k\in\N$, there exists $h_k\in\D$ such that $h_k\ge 1$ $\mu$-a.e.\ on $E_k$.
\item The set $\bigcup_{k=1}^\infty \D_{E_k}$ is dense in $\D$.
\end{enumerate}
\end{defn}
\begin{rem}
We note that measurable nests exist, from \cite[Lemma~3.1]{AH}. 
For every $k \in \N$, we have $\mu(E_{k})<\infty$ because of condition (i). By condition (ii), $\mu(E \setminus \bigcup_{k=1}^{\infty} E_{k})=0$ follows.
If $\{E_k\}_{k=1}^\infty$ and $\{E'_k\}_{k=1}^\infty$ are both measurable nests, then so is $\{E_k\cap E'_k\}_{k=1}^\infty$, from \cite[Lemma~3.2]{AH}.
\end{rem}
\begin{defn}
For a measurable nest $\{E_k\}_{k=1}^\infty$ and $p\in[1,\infty]$, we set
\begin{align*}
L^p_\loc(\mu,\{E_k\})&=\{f\in L^0(\mu)\mid f\bone_{E_k}\in L^p(\mu)\text{ for every $k\in\N$}\},\\
\D_{\loc}(\{E_k\})&=\left\{f\in L^0(\mu)\;\middle|\;\parbox{0.39\hsize}{%
There exist $\{f_k\}_{k=1}^\infty\subset\D$ such that $f=f_k$ $\mu$-a.e.\ on $E_k$ for each $k\in\N$}\,\right\},\\
\D_{\loc,b}(\{E_k\})&=\D_{\loc}(\{E_k\})\cap L^\infty(\mu),\\
\D_{\loc,b,+}(\{E_k\})&=\left\{ f \in\D_{\loc,b}(\{E_k\})  \;\middle|\; f \ge 0 \ \mu\text{-a.e.} \right\}.
\end{align*}
\end{defn}
We remark that 
$L^p(\mu)\subset L^p_\loc(\mu,\{E_k\})\subset L^q_\loc(\mu,\{E_k\})$ 
for $1\le q\le p\le \infty$.
For $f\in \D_{\loc}(\{E_k\})$, $\nb f$ is defined as an $H$-valued measurable function on $E$ by $\nb f=\nb f_k$ on $E_k$, where $f_k\in\D$ and $f_k=f$ $\mu$-a.e.\ on $E_k$.
From Proposition~\ref{prop:D}, $\nb f$ is well-defined up to $\mu$-equivalence.
\begin{defn}
For a measurable nest $\{E_k\}_{k=1}^\infty$, we set
\[
\D_0(\{E_k\})=\bigl\{f\in \D_{\loc,b}(\{E_k\})\bigm| |\nb f|_H\le 1\ \mu\text{-a.e.}\bigr\}.
\]
\end{defn}
This definition is consistent with \cite[Definition~2.6]{AH},
which considers more general situations. 
The function space $\D_0(\{E_k\})$ does not depend on the choice of $\{E_k\}_{k=1}^\infty$, from \cite[Proposition~3.9]{AH}; 
we therefore denote it as $\D_0$.
We now define the intrinsic distance between two sets as follows.
\begin{defn}[{%
see \cite[p.~1241]{AH}}]\label{def:intrinsic}
For $A,B\in\cB$ with positive $\mu$ measures, we define
\[
\sd(A,B)=\sup_{f\in \D_0}\left\{\essinf_{x\in A}f(x)-\esssup_{x\in B}f(x)\right\}\in[0,\infty],
\]
where the essential infimum $\essinf$ and essential supremum $\esssup$ are taken with respect to $\mu$.
\end{defn}
We introduce the concept of a distance-like function $\sd_B$ from the set $B$ and quote a result from \cite{AH}.
For $B\in\cB$ and $N\ge0$, define
\[
\D_{B,N}=\{f\in \D_0\mid f=0\text{ on $B$ and }0\le f\le N\ \mu\text{-a.e.}\}.\footnote{This definition is slightly different from that in \cite{AH}, but the difference is unimportant in our context.}
\]
For $x,y\in\R$, we let $x\vee y$ and $x\wg y$ denote $\max\{x,y\}$ and $\min\{x,y\}$, respectively.
\begin{prop}[{\cite[Proposition~3.11]{AH}}]\label{prop:distance}
For each $B\in \cB$, there exists a unique $[0,+\infty]$-valued measurable function $\sd_B$ on $E$ (up to $\mu$-null sets) such that, for every $N>0$, $\sd_B\wg N$ is the maximal element of $\D_{B,N}$: $\sd_B\wg N\in \D_{B,N}$ and $f\le\sd_B\wg N$ $\mu$-a.e.\ for every $f\in \D_{B,N}$. Moreover, $\sd(A,B)=\essinf_{x\in A}\sd_B(x)$ for every $A\in \cB$.
\end{prop}
We define $\nb\sd_B$ by $\nb\sd_B=\nb(\sd_B\wg N)$ on $\{\sd_B\le N\}$ for $N>0$ and $\nb\sd_B=0$ on $\{\sd_B=\infty\}$. This is well-defined, by Proposition~\ref{prop:D}.

To introduce the perturbation terms, let $b$ and $c$ be $H$-valued measurable functions on $E$, and let $V$ be a real measurable function on $E$. 
From here, we always assume the following minimal requirement.
\begin{assumption}\label{as:1}
\begin{enumerate}[({A}.1)]
\item There exists a measurable nest $\{E_k\}_{k=1}^\infty$ such that 
$|b|_H,|c|_H\in L^2_\loc(\mu,\{E_k\})$ and $V\in L^1_\loc(\mu,\{E_k\})$.
\item There exist $\eta \in [0,1)$, $\theta\ge0$, $\om\ge0$, and $l\ge 0$ such that, for every $f,g\in \bigcup_{k=1}^\infty\D_{E_k,b}$,
\begin{equation}\label{eq:A2_1}
-\int_E  \{(b+c,\nb f)_H f+V f^2\}\,d\mu \le \eta \cE^0(f)+\theta\|f\|_2^2
\end{equation}
and
\begin{equation}\label{eq:A2_2}
\left|\int_E \bigl\{ (b,\nb f)_H g+(c,\nb g)_H f + V fg\bigr\}\,d\mu\right|
\le \om \cE^0_l(f)^{1/2}\cE^0_l(g)^{1/2}.
\end{equation}
\end{enumerate}
\end{assumption}
For $f,g\in \bigcup_{k=1}^\infty\D_{E_k,b}$, we define
\begin{equation}\label{eq:cE}
\cE(f,g)=\cE^0(f,g)+\int_E \bigl\{ (b,\nb f)_H g+(c,\nb g)_H f+ V fg\bigr\}\,d\mu.
\end{equation}
It follows that
\begin{align*}
\cE(f,f)+\theta\|f\|_2^2
&= \cE^0(f)+\int_E \bigl\{ (b,\nb f)_H f+(c,\nb f)_H f+ V f^2\bigr\}\,d\mu+\theta\|f\|_2^2\\
&\ge(1-\eta)\cE^0(f),\\
|\cE(f,g)|
&\le |\cE^0(f,g)|+\left|\int_E \bigl\{ (b,\nb f)_H g+(c,\nb g)_H f+ V fg\bigr\}\,d\mu\right|\\
&\le (1+\om)\cE^0_l(f)^{1/2}\cE^0_l(g)^{1/2},
\end{align*}
and $\bigcup_{k=1}^\infty\D_{E_k,b}$ is dense in $\D$.
Therefore, $\cE(\cdot,\cdot)$ extends continuously to a bilinear form on $\D$ and the bilinear form 
$\D\times\D\ni(f,g)\mapsto \cE_\theta(f,g):=\cE(f,g)+\theta\int_E fg\,d\mu$ 
is a coercive closed form on $L^2(\mu)$.
Thus, a strongly continuous semigroup $\{T_t\}_{t>0}$ exists on $L^2(\mu)$ and some closed operator $(\cL,\Dom(\cL))$ on $L^2(\mu)$ associated with $(\cE,\D)$ satisfies $\cE(f,g)=-\int_E (\cL f)g\,d\mu$ for $f\in \Dom(\cL)$ and $g\in\D$.
In particular, $T_t$ can be given as $e^{\theta t}T^{(\theta)}_t$, where $\{T^{(\theta)}_t\}_{t>0}$ is the semigroup associated with $(\cE_\theta,\D)$.
Formally, $\cL$ is described as $-(1/2)\nb^*\nb-(b,\nb\cdot)_H-\nb^*(c\cdot)-V\cdot$, where $\nb^*$ denotes the adjoint operator of $\nb$. 
Confirming that $(\cE_\theta,\D)$ satisfies the condition~(S) in \cite[Proposition~1.2]{MR2} by an argument similar to \cite[Proof of Theorem~2.2]{MR2} and applying \cite[Theorem~1.5]{MR2},
we see that $\{T_t\}_{t>0}$ is positivity preserving. That is, $T_t f\ge0$ $\mu$-a.e.\ if $f\ge0$ $\mu$-a.e. In general, $\{T_t\}_{t>0}$ is 
not necessarily Markovian.
Let $\hat T_t$ denote the adjoint operator of $T_t$ on $L^2(\mu)$. 
Then, $\{\hat T_t\}_{t>0}$ is also a positivity-preserving semigroup and is associated with a bilinear form $(\hat\cE,\D)$ defined by
\[
  \hat\cE(f,g)=\cE(g,f),\qquad f,g\in\D.
\]

Let $\cB_0$ denote the set of all sets $A\in\cB$ such that $0<\mu(A)<\infty$.
For $A,B\in\cB_0$, define
\[
  P_t(A,B)=\int_A T_t\bone_B\,d\mu \quad\text{for } t>0.
\]
If $\sd(A,B)=\infty$, then the situation is simple, with the proof of the following proposition given in Section~5.
\begin{prop}\label{prop:infty}
Suppose $A,B\in\cB_0$ satisfy $\sd(A,B)=\infty$.
Then, under Assumption~\ref{as:1},
$P_t(A,B)=0$ for all $t>0$. 
In particular, it follows that
\[
\lim_{t\to0}t\log P_t(A,B)= -\frac{\sd(A,B)^2}2\,(=-\infty).
\]
\end{prop}
From this proposition, it is sufficient to consider the case when $\sd(A,B)<\infty$.
For this case, we need some extra assumptions to begin.
Let $\log^\pm x$ denote $0\vee(\pm\log x)$ for $x\ge0$.
\begin{theorem}[upper estimate]\label{th:upper}
Let $A,B\in\cB_0$ with $\sd(A,B)<\infty$. 
Suppose Assumption~\ref{as:1} and the following.
\begin{enumerate}[({B}.1)$_{A,B}$]
\item[\rm(B.1)] There exists $\kp>0$ such that
\begin{equation}\label{eq:upperB1}
\left|\int_E (b-c,\nb f)_H f\,d\mu\right|\le \kp\cE^0_1(f),\quad f\in\bigcup_{k=1}^\infty \D_{E_k,b}.
\end{equation}
\item[\rm(B.2)$_{A,B}$] There exist $\gm\ge0$ and nonnegative numbers $\{\lm_\eps\}_{\eps>0}$ such that 
\begin{equation}\label{eq:lmeps}
\lim_{\eps\to0}\eps \lm_\eps=0
\end{equation}
and
\begin{align}
\int_{\{0<\sd_B<\sd(A,B)\}}&(b-c,\nb\sd_B)_H f^2\,d\mu\le \eps\cE^0(f)+\lm_\eps+\gm\left(\int_E f^2\log^+ f^2\,d\mu\right)^{1/2}\notag\\
&\quad\text{for any $\eps>0$ and $f\in\bigcup_{k=1}^\infty \D_{E_k,b}$ with $\|f\|_2=1$}.
\label{eq:B2}
\end{align}
\end{enumerate}
Then,
\begin{equation}\label{eq:upper}
\varlimsup_{t\to0}t\log P_t(A,B)\le -\frac{\sd(A,B)^2}2.
\end{equation}
\end{theorem}
\begin{theorem}[lower estimate]\label{th:lower}
Let $B\in\cB_0$ and $N>0$.
Suppose Assumption~\ref{as:1} and suppose \eqref{eq:upper} of Theorem~\ref{th:upper} holds for any $A\in\cB_0$ with $\sd(A,B)<N$.
Then, 
\begin{equation}\label{eq:lower}
\varliminf_{t\to0}t\log P_t(A,B)\ge -\frac{\sd(A,B)^2}2
\quad \text{for any $A\in\cB_0$ with $\sd(A,B)<N$.}
\end{equation}
Accordingly, 
\begin{equation}\label{eq:limit}
\lim_{t\to0}t\log P_t(A,B)= -\frac{\sd(A,B)^2}2
\end{equation}
for such $A\in\cB_0$.
\end{theorem}
In particular, we have the following theorem. 
\begin{theorem}
Suppose Assumption~\ref{as:1},  {\rm(B.1)} in Theorem~\ref{th:upper}, and the following.
\begin{enumerate}[(B.2$'$)]
\item There exist $\gm\ge0$ and nonnegative numbers $\{\lm_\eps\}_{\eps>0}$ such that $\lim_{\eps\to0}\eps \lm_\eps=0$ and
\[
\int_E |b-c|_H f^2\,d\mu\le \eps\cE^0(f)+\lm_\eps+\gm\left(\int_E f^2\log^+ f^2\,d\mu\right)^{1/2}
\]
for any $\eps>0$ and $f\in\bigcup_{k=1}^\infty \D_{E_k,b}$ with $\|f\|_2=1$.\end{enumerate}
Then, \eqref{eq:limit} holds for any $A,B\in\cB_0$.
\end{theorem}
\begin{proof}
Since $|\nb \sd_{B}|_H \le 1$ $\mu$-a.e., 
(B.2$'$) implies (B.2)$_{A,B}$ for all $A,B\in\cB_0$ with $\sd(A,B)<\infty$.
The claim follows from Proposition~\ref{prop:infty}, Theorem~\ref{th:upper}, and Theorem~\ref{th:lower}. 
\end{proof}

\begin{rem}\label{rem:2}
\begin{enumerate}
\item Assumption~\ref{as:1}, (B.1), and (B.2$'$) are symmetric with respect to $b$ and $c$, while (B.2)$_{A,B}$ is not.
\item The integral in \eqref{eq:upperB1} is formally rewritten as $\int_E (D^*(b-c)/2)f^2\,d\mu$, which might be easier to understand.
\item We remark that we need only \eqref{eq:upper}, rather than (B.1) and (B.2)$_{A,B}$, for Theorem~\ref{th:lower}.
\end{enumerate}
\end{rem}
A sufficient condition for (B.2$'$) is given as follows.
\begin{prop}\label{prop:B2}
Suppose that $b$ and $c$ are decomposed into $b=b_1+b_2$ and $c=c_1+c_2$ such that  $b_1,b_2,c_1,c_2$ are measurable  
and the following hold.
\begin{enumerate}
\item There exist nonnegative numbers $\{\hat\lm_\eps\}_{\eps>0}$ such that $\lim_{\eps\to0}\eps \hat\lm_\eps=0$ and 
\begin{equation}\label{eq:upperB1'}
\int_E |b_1-c_1|_H f^2\,d\mu\le \eps\cE^0(f)+\hat\lm_\eps\|f\|_2^2,\quad \eps>0,\ f\in\bigcup_{k=1}^\infty \D_{E_k,b}.
\end{equation}
\item There exists $\dl>0$ such that $\exp({\dl|b_2-c_2|_H^2})-1\in L^1(\mu)$.
\end{enumerate}
Then, {\rm (B.2$'$)} holds.
\end{prop}
The proof is based on a simple application of a type of Hausdorff--Young inequality. 
We provide the proof in Section~5 together with a discussion of other sufficient conditions.

\section{Proof of Theorem~\protect{\ref{th:upper}}}
\subsection{$L^p$ property of semigroups}
The following proposition is interesting in its own right, as well as being used in the proof of Theorem~\ref{th:upper}.
Although claims of the kind made by the proposition have been studied in many papers (e.g., \cite{L,SoV,FK} and the references therein), we give a proof for completeness since our framework is slightly different from that used in other proofs.
\begin{prop}\label{prop:Lp}
Suppose Assumption~\ref{as:1} and the following.
There exists $\kp>0$ such that
\begin{equation}\label{eq:upperB1''}
\int_E (b-c,\nb f)_H f\,d\mu\le \kp\cE^0_1(f),\quad f\in\bigcup_{k=1}^\infty \D_{E_k,b}.
\end{equation}
Then, by letting
\begin{equation}\label{eq:q0}
q_0=\frac{\kp+2+\sqrt{\kp^2+4(1-\eta)}}{\kp+\eta}\,(>2)\quad\text{and}\quad
 q'_0=\frac{q_0}{q_0-1},
\end{equation}
$\{T_t|_{L^2(\mu)\cap L^p(\mu)}\}_{t>0}$ (resp.,\ $\{\hat T_t|_{L^2(\mu)\cap L^p(\mu)}\}_{t>0}$) extends to a strongly continuous semigroup on $L^p(\mu)$ for all $p\in[2,q_0]$ (resp.,\ $p\in[q'_0,2]$).
Moreover, the operator norm of the semigroup on $L^p(\mu)$ at $t>0$ is dominated by $\exp\left\{t(\theta+\kp|1-2/p|)\right\}$.
\end{prop}
\begin{proof}
It suffices to consider $\{T_t\}_{t>0}$; the claim for $\{\hat T_t\}_{t>0}$ follows by considering the adjoint semigroup of $\{T_t\}_{t>0}$.

Let $p\ge2$ and $M\ge1$. 
Define a $C^2$-function $\Lambda$ on $[0,\infty)$ so that $\Lambda(0)=\Lambda'(0)=0$ and $\Lambda''(x)=p(p-1)(x\wg M)^{p-2}$ for $x\ge0$.
Also, define the following functions on $[0,\infty)$:
\[
\hat\Lambda(x)=\int_0^x\sqrt{\Lambda''(s)}\,ds,\quad
\tilde\Lambda(x)=\sqrt{\Lambda'(x)x},\quad\text{and}\quad
\check\Lambda(x)=\sqrt{\Lambda'(x)x-2\Lambda(x)}.
\]
Then, by long but straightforward calculation, we can confirm the following inequalities:
\begin{equation}\label{eq:Lp1}
\begin{aligned}
&\Lambda'(x)\ge\frac{p}{2(p-1)}\hat\Lambda(x)\hat\Lambda'(x),\quad
0\le \tilde\Lambda'(x)\le\frac{p}{2\sqrt{p-1}}\hat\Lambda'(x),\\
&0\le \check\Lambda'(x)\le\sqrt{\frac{p(p-2)}{4(p-1)}}\hat\Lambda'(x).
\end{aligned}
\end{equation}
Since $\hat\Lambda(0)=\tilde\Lambda(0)=\check\Lambda(0)=0$, this implies, in particular, 
\begin{equation}\label{eq:Lp2}
\Lambda(x)\ge\frac{p}{4(p-1)}\hat\Lambda(x)^2,\quad
0\le\tilde\Lambda(x)\le\frac{p}{2\sqrt{p-1}}\hat\Lambda(x),\quad\text{and}\quad
0\le\check\Lambda(x)\le\sqrt{\frac{p(p-2)}{4(p-1)}}\hat\Lambda(x).
\end{equation}
Let $C\ge0$ and $t>0$.
Take $f\in L^2(\mu)\cap L^p(\mu)$ with $f\ge0$ $\mu$-a.e.
Since $T_t f\ge0$ $\mu$-a.e.\ and 
$\Lambda''$ is bounded,
we have that $\Lambda'(T_t f)\in L^2(\mu)$, $\Lambda(T_t f)\in L^1(\mu)$ and
\begin{align*}
\frac{d}{dt}\int_E e^{-Ct}\Lambda(T_t f)\,d\mu
&=\int_E\left(-C e^{-Ct}\Lambda(T_t f)+e^{-Ct}\Lambda'(T_t f)\cL T_t f\right)d\mu\\
&=e^{-Ct}\left(-C\int_E \Lambda(T_t f)\,d\mu -\cE(T_t f,\Lambda'(T_t f))\right).
\end{align*}
Moreover, for $v\in\bigcup_{k=1}^\infty\D_{E_k,b,+}$,
\begin{align*}
\cE(v,\Lambda'(v))
&=\frac12\int_E \bigl(\nb v,\nb(\Lambda'(v))\bigr)_H\,d\mu
+\int_E (b,\nb v)_H \Lambda'(v)\,d\mu\\
&\qad+\int_E \bigl(c,\nb(\Lambda'(v))\bigr)_H v\,d\mu
+\int_E V v\Lambda'(v)\,d\mu\\
&=\frac12\int_E \bigl|\nb(\hat\Lambda(v))\bigr|_H^2
+\int_E\bigl(b+c,\nb(\tilde\Lambda(v))\bigr)_H \tilde\Lambda(v)\,d\mu\\
&\qad-\int_E\bigl(b-c,\nb(\check\Lambda(v))\bigr)_H \check\Lambda(v)\,d\mu
+\int_E V\tilde\Lambda(v)^2\,d\mu\\
&\ge \cE^0(\hat\Lambda(v))
-\eta\cE^0(\tilde\Lambda(v))-\theta\|\tilde\Lambda(v)\|_2^2
-\kp\cE^0_1(\check\Lambda(v)).
\quad\text{(from (A.2) and \eqref{eq:upperB1''})}
\end{align*}
By using \eqref{eq:Lp1}, \eqref{eq:Lp2}, and the derivation property 
\eqref{eq:chain}
of $\nb$, we have
\begin{align}
-C\int_E\Lambda(v)\,d\mu-\cE(v,\Lambda'(v))
&\le -\frac{Cp}{4(p-1)}\|\hat\Lambda(v)\|_2^2
-\cE^0(\hat\Lambda(v))+\frac{\eta p^2}{4(p-1)}\cE^0(\hat\Lambda(v))\notag\\*
&\qad+\frac{\theta p^2}{4(p-1)}\|\hat\Lambda(v)\|_2^2+\frac{\kp p(p-2)}{4(p-1)}\cE^0_1(\hat\Lambda(v))\notag\\
&=\frac{-4(p-1)+\eta p^2+\kp p(p-2)}{4(p-1)}\cE^0(\hat\Lambda(v))\notag\\*
&\qad+\frac{p\{-C+\theta p+\kp(p-2)\}}{4(p-1)}\|\hat\Lambda(v)\|_2^2.
\label{eq:Lp3}
\end{align}
Therefore, if $
-4(p-1)+\eta p^2+\kp p(p-2)\le0$---or, more specifically, if $p\in[2,q_0]$ with $q_0$ given by \eqref{eq:q0}---then
the right-hand side of \eqref{eq:Lp3} is non-positive by letting 
$C=C_p:=\theta p+\kp(p-2)$.
Thus, the inequality 
\[
-C_p\int_E\Lambda(v)\,d\mu-\cE(v,\Lambda'(v))\le 0
\]
is valid for all $v\in \D$ with $v\ge0$ $\mu$-a.e.,
by approximating $v$ by elements of $\bigcup_{k=1}^\infty\D_{E_k,b,+}$.
In particular,
\[
\frac{d}{dt}\int_E e^{-C_pt}\Lambda(T_t f)\,d\mu\le0,
\]
which implies 
\[
\int_E e^{-C_pt}\Lambda(T_t f)\,d\mu\le \int_E \Lambda(f)\,d\mu,\quad t>0.
\]
Letting $M\to\infty$, we obtain that $\int_E e^{-C_pt}(T_t f)^p\,d\mu\le \int_E f^p\,d\mu$, from the monotone convergence theorem.
Thus, $\{T_t|_{L^2(\mu)\cap L^p(\mu)}\}_{t>0}$ extends to a semigroup on $L^p(\mu)$ 
that satisfies
$\|T_t\|_{L^p(\mu)\to L^p(\mu)}\le e^{C_pt/p}$.
The strong continuity of the semigroup follows from the result given in \cite{Vo}. (Indeed, it is easy to see that $\{T_t\}_{t>0}$ is a weakly continuous semigroup on $L^p(\mu)$, which implies strong continuity.)
\end{proof}
\begin{cor}\label{cor:Lp}
Suppose Assumption~\ref{as:1} and {\rm(B.1)}.
Then, $\{T_t|_{L^2(\mu)\cap L^p(\mu)}\}_{t>0}$ and $\{\hat T_t|_{L^2(\mu)\cap L^p(\mu)}\}_{t>0}$ extend to strongly continuous semigroups on $L^p(\mu)$ for all $p\in[q_0',q_0]$, where $q_0$ and $q'_0$ are as given in \eqref{eq:q0}.
Moreover, the operator norm of the semigroups on $L^p(\mu)$ at $t>0$ are dominated by $\exp\left\{t(\theta+\kp|1-2/p|)\right\}$.
\end{cor}
\begin{proof}
Apply Proposition~\ref{prop:Lp} to $(\cE,\D)$ and $(\hat\cE,\D)$.
\end{proof}

\subsection{Preliminary estimates}
In this subsection, we provide several quantitative estimates used in the proof of Theorem~\ref{th:upper}.

We take a non-decreasing $C^2$-function $\xi$ on $[0,\infty)$ such that $\xi(x)=(x-1)^3\vee0$ for $x\in[0,3/2]$ and $\xi(x)=1$ for $x\in[2,\infty)$.
Define
\[
\zt(x,y)=2x+(y x^{y-1}-2x)\xi(x),\qquad x\ge0,\ y\ge2,
\]
and
\begin{equation}\label{eq:tu}
\tau(x,y)=\int_0^x \zt(s,y)\,ds,\qquad x\ge0,\ y\ge2.
\end{equation}
Then, 
\[
\del_x \zt(x,y)=2+(y(y-1)x^{y-2}-2)\xi(x)+(yx^{y-1}-2x)\xi'(x)\ge2.
\]
For $R>2$, we define the following functions on $[0,\infty)\times [2,\infty)$:
\begin{align*}
g_R(x,y)&=\sqrt{(\del_x \zt)(x\wg R,y)},&
h_R(x,y)&=\int_0^x g_R(s,y)^2\,ds,\\
\ph_R(x,y)&=\int_0^x h_R(s,y)\,ds,&
\rho_R(x,y)&=\int_0^x g_R(s,y)\,ds,\\
\iota_R(x,y)&=x h_R(x,y)-2\ph_R(x,y),&
\psi_R(x,y)&=\sqrt{x h_R(x,y)}.
\end{align*}
\begin{lemma}\label{lem:4}
For any fixed $y\ge2$, the following hold.
\begin{enumerate}
\item For $x\ge0$,
\begin{align*}
h_R(x,y)
&=\zt(x\wg R,y)+\bigl((x-R)\vee0\bigr)\del_x \zt(R,y)\\
&=\zt(x\wg R,y)+\bigl((x-R)\vee0\bigr)y(y-1)R^{y-2}.
\end{align*}
In particular,
$h_R(x,y)=\zt(x,y)$ if $0\le x \le R$.
\item $2x\le h_R(x,y)\le\zt(x,y)\le \max\{2x,y x^{y-1}\}$ 
for all $R>2$ and $x\ge 0$.
\item For $x\ge 0$, $h_R(x,y)$ converges to $\zt(x,y)$ as $R\to\infty$.
\item $x^2\le \ph_R(x,y)\le \tau(x,y)\le\max\{x^2,x^y\}$ 
for all $R>2$ and $x\ge 0$.
\item $g_R(x,y)\ge\sqrt2$ for $x\ge0$ and $g_R(x,y)\ge\sqrt{y(y-1)}(x\wg R)^{y/2-1}$ for $x\ge2$.
\end{enumerate}
\end{lemma}
\begin{proof} 
(i) and (v): Straightforward from the definitions.

\noindent(ii): The first and last inequalities are easy to prove. Since
\[
\del_x h_R(x,y)
=\del_x \zt(x\wg R,y)
\le \del_x \zt(x,y),
\]
the second inequality also holds.

\noindent(iii): This follows from (i).

\noindent(iv): This follows by integrating each term of the inequality in (ii).
\end{proof}
\begin{lemma}\label{lem:pos}
For any $x\ge0$ and $y\ge2$, the following hold.
\begin{enumerate}
\item $xg_R(x,y)^2\ge h_R(x,y)$.
\item $\iota_R(x,y)\ge0$.
\item $\rho_R(x,y)\le x g_R(x,y)$.
\end{enumerate}
\end{lemma}
\begin{proof}
(i): For $x\in[0,R]$, 
\begin{align}\label{eq:xgR}
xg_R(x,y)^2-h_R(x,y)
&=x\del_x \zt(x,y)-\zt(x,y)\notag\\
&= y(y-2)x^{y-1}\xi(x)+(yx^y-2x^2)\xi'(x),
\end{align}
which is nonnegative since $\xi'(x)=0$ for $x\le1$.
For $x>R$, 
\begin{align}
xg_R(x,y)^2-h_R(x,y)
&=x\del_x \zt(R,y)-\left\{\zt(R,y)-(x-R)\del_x \zt(R,y)\right\}\notag\\
&=y(y-2)R^{y-1}\ge0.\label{eq:xgR2}
\end{align}

\noindent(ii): This follows from identities $\iota_R(0,y)=0$, ${\del_x}\iota_R(x,y)=xg_R(x,y)^2-h_R(x,y)$, and (i).

\noindent(iii): By the Cauchy--Schwarz inequality,
\begin{align*}
\rho_{R}(x,y)&=\int_0^x g_R(s,y)\,ds
\le\left(x\int_0^x g_R(s,y)^2\,ds\right)^{1/2}\\
&=(x h_R(x,y))^{1/2}
\le x g_R(x,y).\quad\text{(from (i))} \qedhere
\end{align*}
\end{proof}
\begin{lemma}\label{lem:small}
For $y\in[2,3]$, $\sqrt{\iota_R(x,y)}$ is continuously differentiable with respect to $x$.
Moreover, there exists some $K_0>0$ independent of $x$ and $y$ such that
\begin{align}
0&\le {\del_x}\sqrt{\iota_R(x,y)}\le K_0 \sqrt{y-2}g_R(x,y)
\label{eq:small1}\\
\shortintertext{and}
0&\le \sqrt{\iota_R(x,y)}\le K_0 \sqrt{y-2}\rho_R(x,y).
\label{eq:small2}
\end{align}
\end{lemma}
\begin{proof}
Since $\iota_R(x,2)\equiv0$, it suffices to consider the case $y\in(2,3]$. The continuous differentiability of $\sqrt{\iota_R(x,y)}$
is trivial for $x\ne1$ since $\iota_R(x,y)=0$ for $x\in[0,1]$ and $\iota_R(x,y)>0$ for $x>1$.
For $x\in(1,R]$, combining \eqref{eq:xgR} and \eqref{eq:xgR2},
\[
{\del_x}\iota_R(x,y)=
\begin{cases}
(y-2)\left[ y x^{y-1}\xi(x)+\left\{x^y+\dfrac{2(x^y-x^2)}{y-2}\right\}\xi'(x)\right]
& (1<x\le R),\\
y(y-2)R^{y-1} & (x> R).
\end{cases}
\]
Thus, we can confirm that there exist $K_1$ and $K_2$ with $0<K_1<K_2$ such that 
\[
K_1(y-2)\hat\iota(x,y)\le {\del_x}\iota_R(x,y) \le K_2(y-2)\hat\iota(x,y)\quad
\text{for $x\in(1,R]$ and $y\in(2,3]$},
\]
where $\hat\iota$ is defined as
\[
\hat\iota(x,y)=(x-1)^2 \bone_{(1,3/2]}(x)+\bone_{(3/2,2]}(x)+x^{y-1}\bone_{(2,R]}(x).
\]
For $x\in(1,3/2]$, we have 
\[
K_1(x-1)^3(y-2)/3\le \iota_R(x,y)\le K_2(x-1)^3(y-2)/3.
\]
Thus, $\lim_{x\downarrow1}\sqrt{\iota_R(x,y)}/(x-1)=0$. That is, $\sqrt{\iota_R(x,y)}$ is continuously differentiable with respect to $x$ at $1$.
Furthermore,
\[
0<\frac{\del_x\iota_R(x,y)}{\sqrt{\iota_R(x,y)}}\le K_2\sqrt{3/K_1}\sqrt{x-1}\sqrt{y-2}
\le K_2 K_1^{-1/2}\sqrt{y-2}g_R(x,y)
\]
from Lemma~\ref{lem:4}(v).

For $x\in(3/2,2]$, we have $\iota_R(x,y)\ge \iota_R(3/2,y)\ge K_1(y-2)/24$ and
\[
0<\frac{\del_x\iota_R(x,y)}{\sqrt{\iota_R(x,y)}}
\le K_2 \sqrt{24/K_1}\sqrt{y-2}\le K_2 \sqrt{12/K_1}\sqrt{y-2}g_R(x,y).
\] 
For $x\in(2,R]$, 
\begin{align*}
\iota_R(x,y)&\ge \iota_R(2,y)+\int_2^R \del_x\iota_R(s,y)\,ds 
\ge \frac{K_1}{24}(y-2)+K_1(y-2)\frac{x^{y}-2^y}{y}\\
&\ge \frac{K_1}{24}(y-2)\left(1+\frac{x^{y}-2^y}{8}\right)
\ge K_3(y-2)x^y
\end{align*}
for  $K_3=K_1/192$, and
\[
0<\frac{\del_x\iota_R(x,y)}{\sqrt{\iota_R(x,y)}}
\le K_2K_3^{-1/2}\sqrt{y-2}x^{y/2-1}\le K_2K_3^{-1/2}\sqrt{y-2}g_R(x,y).
\] 
For $x>R$, we have $\iota_R(x,y)\ge \iota_R(R,y)\ge K_3 (y-2)R^y$ and
\[
0<\frac{\del_x\iota_R(x,y)}{\sqrt{\iota_R(x,y)}}
\le K_3^{-1/2}\sqrt{y-2}R^{y/2-1}
\le\sqrt{2/K_3}\sqrt{y-2}g_R(x,y).
\]
From these estimates, \eqref{eq:small1} holds by setting 
$K_0=2^{-1}\max\{K_2 \sqrt{12/K_1}, K_2K_3^{-1/2}, \sqrt{2/K_3}\}$.
Integrating each term of \eqref{eq:small1} gives \eqref{eq:small2}.
\end{proof}

\begin{lemma}\label{lem:6''}
For $x\ge0$ and $y\in[2,3]$,
\begin{equation}\label{eq:lemextra}
\max\{0,x^y\log^+ x-(2\log 2)x^2\}\le {\del_y\tau}(x,y)\le x^y\log^+ x.
\end{equation}
\end{lemma}
\begin{proof}
We have
\[
{\del_y\tau}(x,y)
= \int_0^x {\del_y\zt}(s,y)\,ds
= \int_0^x \left(s^{y-1}+ys^{y-1}\log s\right)\xi(s)\,ds.
\]
From $\xi(x)=0$ for $x\in[0,1]$, it follows that ${\del_y\tau}(x,y)$ vanishes for $x\in[0,1]$ and is non-decreasing in $x$. For $x\in[0,2]$ and $y\in[2,3]$,
\begin{align*}
x^y \log^+ x-(2\log 2)x^2
\le x^2(2^{y-2}\log 2-2\log 2)\le0\le {\del_y\tau}(x,y).
\end{align*}
Moreover, for $x\ge2$ and $y\in[2,3]$,
\begin{align*}
{\del_y\tau}(x,y)
&\ge \int_2^x \left(s^{y-1}+ys^{y-1}\log s\right)ds\\
&=s^y \log s|_{s=2}^{s=x}\\
&=x^y \log x-2^y \log 2\\
&\ge x^y \log^+ x-(2\log 2)x^2.
\end{align*}
Therefore, the first inequality of \eqref{eq:lemextra} holds.
For $x>1$, we have
\[
{\del_y\tau}(x,y)
\le \int_1^x \left(s^{y-1}+ys^{y-1}\log s\right)ds
=s^y \log s|_{s=1}^{s=x}
=x^y \log x.
\]
Thus, the second inequality of \eqref{eq:lemextra} holds.
\end{proof}
\begin{lemma}\label{lem:5}
For each $\eps>0$, there exists some $y_0=y_0(\eps)>2$ such that $g_R(x,y)x^{-\eps}$ is non-increasing in $x$ for any $y\in[2,y_0]$ and $R>2$.
\end{lemma}
\begin{proof}
Since $g_R(x,y)=g_R(R,y)$ for $x\ge R$, the term $g_R(x,y)x^{-\eps}$ is always non-increasing for $x\ge R$.
It therefore suffices to consider only $x$ in $[0,R]$.
We may additionally assume that $\eps\in(0,1/2)$.
Define
\[
\nu(x,y):=\bigl(g_R(x,y)x^{-\eps}\bigr)^2, 
\quad (x,y)\in[0,R]\times[2,\infty).
\]
It suffices to prove that there exists some $y_0(\eps)>2$ such that, for $(x,y)\in[0,R]\times[2,y_0(\eps)]$,
\begin{equation}\label{eq:lem5-1}
{\del_x\nu}(x,y)
\le0.
\end{equation}
By definition, for $(x,y)\in[0,R]\times[2,\infty)$,
\begin{align*}
\nu(x,y)
&={\del_x\zt}(x,y)x^{-2\eps}\\
&=2x^{-2\eps}+\bigl(y(y-1)x^{y-2-2\eps}-2x^{-2\eps}\bigr)\xi(x)+\bigl(y x^{y-1-2\eps}-2x^{1-2\eps}\bigr)\xi'(x)
\end{align*}
and
\begin{align*}
{\del_x\nu}(x,y)
&=-4\eps x^{-1-2\eps}+\bigl(y(y-1)(y-2-2\eps)x^{y-3-2\eps}+4\eps x^{-1-2\eps}\bigr)\xi(x)\\
&\qad+\bigl(2y(y-1-\eps)x^{y-2-2\eps}-4(1-\eps)x^{-2\eps}\bigr)\xi'(x)
+\bigl(y x^{y-1-2\eps}-2x^{1-2\eps}\bigr)\xi''(x).
\end{align*}

Suppose $0\le x\le1$. 
Since $\xi(x)=\xi'(x)=\xi''(x)=0$,
\[
{\del_x\nu}(x,y)=-4\eps x^{-1-2\eps};
\]
thus, \eqref{eq:lem5-1} holds for this case.

Suppose $2\le x\le R$. Since $\xi(x)=1$ and $\xi'(x)=\xi''(x)=0$,
\[
{\del_x\nu}(x,y)=y(y-1)(y-2-2\eps)x^{y-3-2\eps}.
\]
Therefore, \eqref{eq:lem5-1} holds for $y\in[2,2+2\eps]$.

Since $\xi$, $\xi'$, and $\xi''$ are all bounded, ${\del_x\nu}(x,y)$ converges to $-4\eps x^{-1-2\eps}$ uniformly in $x\in[1,2]$ as $y\to2$.
Thus, there exists some $y_0(\eps)\in(2,2+2\eps]$ such that
\begin{equation}\label{eq:lem5}
\left|{\del_x\nu}(x,y)+4\eps x^{-1-2\eps}\right|\le 4\eps\cdot 2^{-1-2\eps},\quad
(x,y)\in[1,2]\times[2,y_0(\eps)].
\end{equation}
Equation \eqref{eq:lem5} implies \eqref{eq:lem5-1} for $(x,y)\in[1,2]\times[2,y_0(\eps)]$. 
\end{proof}
\begin{lemma}\label{lem:6}
The following inequalities hold for $\dl\in(0,1/2)$, $R>2$, $x\ge0$, and $y\in [2,y_0(\dl)]${\rm:}
\begin{enumerate}
\item $x g_R(x,y)\le (1+\dl)\rho_R(x,y)$,
\item $x g_R(x,y)^2 \le (1+2\dl)h_R(x,y)$,
\item $x h_R(x,y)\le (1+\dl)^2\rho_R(x,y)^2$,
\item $x g_R(x,y)^2-h_R(x,y)\le \dfrac{2\dl(1+\dl)}{1+2\dl}\rho_R(x,y) g_R(x,y)$,
\item $\rho_R(x,y)^2\le 2(1+2\dl)\ph_R(x,y)\le2(1+2\dl)\tau(x,y)$,
\item $\left({\del_x}\psi_R(x,y)\right)^2\le(1+2\dl)g_R(x,y)^2$.
\end{enumerate}
\end{lemma}
\begin{proof}
In the following, we omit $y$ from the notation.

From Lemma~\ref{lem:5}, $g_R(s)s^{-\dl}\ge g_R(x) x^{-\dl}\ge0$ if $0<s\le x$. Then, we have
\begin{align*}
\rho_R(x)&=\int_0^x g_R(s)\,ds
\ge g_R(x) x^{-\dl}\int_0^x s^\dl\,ds
=\frac{x g_R(x)}{1+\dl}\\
\shortintertext{and}
h_R(x)&=\int_0^x g_R(s)^2\,ds\ge g_R(x)^2 x^{-2\dl}\int_0^x s^{2\dl}\,ds
=\frac{x g_R(x)^2}{1+2\dl}.
\end{align*}
Thus, (i) and (ii) hold.
Combining (i) and Lemma~\ref{lem:pos}(i) gives (iii).
From (i) and (ii),
\[
xg_R(x)^2-h_R(x)
\le \left(1-\frac{1}{1+2\dl}\right)x g_R(x)^2
\le \frac{2\dl(1+\dl)}{1+2\dl}\rho_R(x) g_R(x),
\]
which proves (iv).
Next, we prove (v).
From (ii) and Lemma~\ref{lem:pos}(iii), we have
\[
{\del_x}\bigl(\rho_R(x)^2\bigr)
=2g_R(x) \rho_R(x)
\le 2x g_R(x)^2
\le 2(1+2\dl)h_R(x).
\]
Then,
\[
\rho_R(x)^2\le 2(1+2\dl)\int_0^x h_R(s)\,ds
=2(1+2\dl)\ph_R(x).
\]
The second inequality of (v) follows from Lemma~\ref{lem:4}(iv).

Last, we prove (vi). 
The  inequality holds for $x=0$ by direct computation.
Let $x>0$. 
By using (ii) and Lemma~\ref{lem:pos}(i), we have
\[
\left({\del_x}\psi_R(x)\right)^2
=\frac{\bigl(x g_R(x)^2+h_R(x)\bigr)^2}{4x h_R(x)}
\le\frac{\left(2x g_R(x)^2\right)^2}{4(1+2\dl)^{-1}x^2 g_R(x)^2}
=(1+2\dl)g_R(x)^2.\qedhere
\]
\end{proof}
\subsection{Derivation of a differential inequality}
In this and further subsections, we prove Theorem~\ref{th:upper}.
The following inequality is often used throughout this paper without specific mention:
\[
|xy|\le \frac1\a x^2+\frac{\a}{4}y^2
\qquad\text{for }\a>0\text{ and }x,y\in \R.
\]

Let $\eps\in\bigl(0,(1-\eta)/6\bigr]$.
Take $\dl>0$ such that 
\begin{equation}\label{eq:dl1}
(1+\dl)^2\le 1+\eps.
\end{equation}
In particular, we use $2\dl\le\eps\le1/6$.
Let 
\[
q=\min\left\{3, 1+\frac{q_0}2, y(\dl), 2+\frac{\eps}{\kp K_0^2}\right\}\in(2,3]
\quad\text{and}\quad S=\frac{3\gm^2}{\eps}.
\]
Here, $q_0$, $y(\cdot)$, and $K_0$ are as provided in 
\eqref{eq:q0}, Lemma~\ref{lem:5}, and Lemma~\ref{lem:small}, respectively.
We set 
\[
t_0=\min\{1,(q-2)/S\}
\quad\text{and}\quad p(t)=q-St,\ \  t\in[0,t_0].
\]
Note that $p(t)\in[2,q]$ for $t\in[0,t_0]$. 

Let $A,B\in\cB_0$ with $\sd(A,B)<\infty$ and set $w=\sd_B\wg \sd(A,B)\in\D_0$.
Let $\a\in\R\setminus\{0\}$ and define 
\begin{equation}\label{eq:sigma}
u_t=T_t\bone_B,\ F(t)=e^{\a w}u_t,\
\text{and }
\sg(t)=\int_E \tau(F(t),p(t))\,d\mu
\quad \text{for }t\in[0,t_0].
\end{equation}
\begin{lemma}\label{lem:7}
The function $\sg$ is continuously differentiable on $(0,t_0]$ and
\begin{align}
\sg'(t)&=\int_E \left\{ {\del_x\tau}(F(t),p(t))F'(t)
+{\del_y\tau}(F(t),p(t))p'(t)\right\}d\mu\notag\\
&=\int_E \zt(F(t),p(t))e^{\a w}\cL u_t\,d\mu-S\int_E{\del_y\tau}(F(t),p(t))\,d\mu.
\label{eq:lem7}
\end{align}
\end{lemma}
\begin{proof}
We justify the formal calculation.
First, let us recall the following fact (see, e.g., \cite[Theorems~21.4 and 21.8]{Ba}): Suppose $r\in[1,\infty)$ and that $\chi\in L^1(\mu)$ satisfies $0<\chi\le1$ $\mu$-a.e.
Then, for functions $\{f_n\}_{n=1}^\infty$ in $L^r(\mu)$ and $f\in L^0(\mu)$, the sequence $f_n$ converges to $f$ in $L^r(\mu)$ if and only if $f_n$ converges to $f$ in measure with respect to $\chi\cdot\mu$ and
\[
\lim_{K\to\infty}\sup_{n}\int_E \{(|f_n|^r-K\chi)\vee0\} \,d\mu=0.
\]
For $t\in(0,t_0]$ and $\{t_n\}_{n=1}^\infty\subset(0,t_0]\setminus\{t\}$ converging to $t$,
\begin{align}
&\bigl(\tau(F(t_n),p(t_n))-\tau(F(t),p(t))\bigr)/(t_n-t)\notag\\
&=\left.\tau\bigl(F(t)+s(F(t_n)-F(t)),p(t)+s(p(t_n)-p(t))\bigr)\bigl|_{s=0}^{s=1}\middle/(t_n-t)\right.\notag\\
&=(t_n-t)^{-1}\int_0^1\left\{(F(t_n)-F(t)){\del_x\tau}(g_{s,n},h_{s,n})
+(p(t_n)-p(t)){\del_y\tau}(g_{s,n},h_{s,n})\right\}ds\notag\\
&=\frac{F(t_n)-F(t)}{t_n-t}\int_0^1\zt(g_{s,n},h_{s,n})\,ds
-S\int_0^1{\del_y\tau}(g_{s,n},h_{s,n})\,ds,
\label{eq:tau}
\end{align}
where $g_{s,n}=F(t)+s(F(t_n)-F(t))\,(\ge0)$ and $h_{s,n}=p(t)+s(p(t_n)-p(t))$.
For each $s\in[0,1]$, $g_{s,n}$ converges to $F(t)$ in $L^r$ as $n\to\infty$ for every $r\in[2,q_0]$, and $h_{s,n}$ converges to $p(t)$ as $n\to\infty$.
In particular, for every $s\in[0,1]$ and $r\in[2,q_0]$, $g_{s,n}$ converges to $F(t)$ in measure with respect to $\chi\cdot\mu$ and
\[
\lim_{K\to \infty} \sup_n \int_E\{((g_{s,n})^r-K\chi)\vee0\}\,d\mu=0.
\]
From the continuity of $\zt$, $\zt(g_{s,n},h_{s,n})$ converges to $\zt(F(t),p(t))$ in measure with respect to $\chi\cdot\mu$ as $n\to\infty$. 
Lemma~\ref{lem:4}(ii) and the inequality $2<2(q-1)<q_0$ together imply $\zt(x,y)^2\le 4x^2+q_0^2 x^{q_0}$ for $x\ge0$ and $y\in[2,q_0]$. Thus,
\begin{align*}
&\sup_n \int_E\{(\zt(g_{s,n},h_{s,n})^2-K\chi)\vee0\}\,d\mu\\
&\le\sup_n \int_E\{(4(g_{s,n})^2-K\chi/2)\vee0\}\,d\mu
+\sup_n \int_E\{(q_0^2(g_{s,n})^{q_0}-K\chi/2)\vee0\}\,d\mu\\
&\to0 \quad\text{as }K\to\infty.
\end{align*}
From the above, $\zt(g_{s,n},h_{s,n})$ converges to $\zt(F(t),p(t))$ in $L^2(\mu)$.
Moreover, since it is easy to see that $\{\zt(g_{s,n},h_{s,n})\}_{s\in[0,1],\,n\in\N}$ is bounded in $L^2(\mu)$, we obtain
\[
\left\|\int_0^1\zt(g_{s,n},h_{s,n})\,ds-\zt(F(t),p(t))\right\|_2
\le\int_0^1\|\zt(g_{s,n},h_{s,n})-\zt(F(t),p(t))\|_2\,ds
\to0
\]
as $n\to\infty$, by the dominated convergence theorem.
Thus, the first term of \eqref{eq:tau} converges to $F'(t)\zt(F(t),p(t))$ in $L^1(\mu)$ as $n\to\infty$ because $(F(t_n)-F(t))/(t_n-t)$ converges to $F'(t)$ in $L^2(\mu)$.

In the same manner, we can prove that $\int_0^1{\del_y\tau}(g_{s,n},h_{s,n})\,ds$ converges to ${\del_y\tau}(F(t),p(t))$ in $L^1(\mu)$ as $n\to\infty$, by using Lemma~\ref{lem:6''}. 
Thus, \eqref{eq:lem7} is proved. 
The proof of the continuity of $\sg'(t)$ proceeds analogously.
\end{proof}

We fix $t\in(0,t_0]$ and estimate the first term of \eqref{eq:lem7}.
Recall the measurable nest $\{E_k\}_{k=1}^\infty$ in Assumption~\ref{as:1}. Take a sequence of functions $\{w_k\}_{k=1}^\infty$ such that for every $k\in\N$, $w_k\in\D_{E_k,b}$, $w_k=w$ $\mu$-a.e.\ on $E_k$ and $0\le w_k\le N$ $\mu$-a.e.
There also exist functions $\{u^{(k)}\}_{k=1}^\infty$ such that $u^{(k)}\in\D_{E_k,b}$ for each $k$ and
the sequence $u^{(k)}$ converges to $u_t$ in $\D$ as $k\to\infty$. By considering $0\vee (u^{(k)}\wg u_t)$ (and the Ces\`aro means if necessary) we may assume that $0\le u^{(k)}\le u_t$ $\mu$-a.e.\ for every $k$ and $\lim_{k\to\infty}u^{(k)}=u_t$ $\mu$-a.e.
Let $F^{(k)}=e^{\a w_k}u^{(k)}\in\D_{E_k,b}$ for each $k$.
Then,
\[
0\le F^{(k)} = e^{\a w}u^{(k)}\le F(t) \quad\mu\text{-a.e.}
\]
and
\begin{align}
\int_E \zt(F(t),p(t))e^{\a w}\cL u_t\,d\mu
&=\int_E \lim_{R\to\infty} h_R(F(t),p(t))e^{\a w}\cL u_t\,d\mu
\qquad\text{(from Lemma~\ref{lem:4})}\notag\\
&=\lim_{R\to\infty}\int_E \lim_{k\to\infty} h_R(F^{(k)},p(t))e^{\a w_k}\cL u_t\,d\mu\notag\\
&=\lim_{R\to\infty}\lim_{k\to\infty} -\cE\bigl(u_t,h_R(F^{(k)},p(t))e^{\a w_k}\bigr).
\label{eq:limlim}
\end{align}
\begin{lemma}\label{lem:8}
For any $R>2$,
\[
\lim_{k\to\infty} \cE\bigl(u_t,h_R(F^{(k)},p(t))e^{\a w_k}\bigr)
=\lim_{k\to\infty} \cE\bigl(u^{(k)},h_R(F^{(k)},p(t))e^{\a w_k}\bigr).
\]
\end{lemma}
\begin{proof}
Since $u^{(k)}$ converges to $u_t$ in $\D$ as $k\to\infty$, proving that the sequence $\{h_R(F^{(k)},p(t))e^{\a w_k}\}_{k=1}^\infty$ is bounded in $\D$ suffices. Boundedness in $L^2(\mu)$ is straightforward. 
For each $k$,
\begin{align}
&\cE^0\bigl(h_R(F^{(k)},p(t))e^{\a w_k}\bigr)^{1/2}\notag\\
&\le \cE^0\bigl(h_R(F^{(k)},p(t))(e^{\a w_k}-1)\bigr)^{1/2}
 +\cE^0\bigl(h_R(F^{(k)},p(t))\bigr)^{1/2}\notag\\
&\le \bigl\|g_R(F^{(k)},p(t))^2(e^{\a w_k}-1)|\nb F^{(k)}|_H\bigr\|_2
 +\bigl\| h_R (F^{(k)},p(t))\a e^{\a w_k}|\nb w_k|_H\bigr\|_2\notag\\
&\qad+\bigl\|g_R(F^{(k)},p(t))^2|\nb F^{(k)}|_H\bigr\|_2.
\label{eq:lem8}
\end{align}
We note that $g_R$ is a bounded function, that $|e^{\a w_k}-1|\le e^{|\a|N}$,
that $h_R(F^{(k)},p(t))=0$ on $E\setminus E_k$, that $|\nb w_k|_H \le1$ on $E_k$ and that
\begin{align*}
\bigl\| |\nb F^{(k)}|_H\bigr\|_2
&\le \bigl\|\a e^{\a w_k}u^{(k)}|\nb w_k|_H\bigr\|_2
  +\bigl\|e^{\a w_k}|\nb u^{(k)}|_H\bigr\|_2\\
&\le |\a| e^{|\a|N}\|u^{(k)}\|_2 + e^{|\a|N}\left(2\cE^0(u^{(k)})\right)^{1/2},
\end{align*}
which is bounded in $k$.
From these estimates, the first and third terms of \eqref{eq:lem8} are bounded in $k$.
Moreover,
Lemma~\ref{lem:4}(ii) and the inequality $2<2(q-1)<q_0$ together imply that
$\{h_R(F^{(k)},p(t))\}_{k=1}^\infty$ is bounded in $L^2(\mu)$.
Thus, the second term of \eqref{eq:lem8} is also bounded in $k$, which completes the proof.
\end{proof}
From this lemma and \eqref{eq:limlim},
\begin{equation}\label{eq:st1}
\int_E \zt(F(t),p(t))e^{\a w}\cL u_t\,d\mu
=\lim_{R\to\infty}\lim_{k\to\infty} -\cE(u^{(k)},h_R(F^{(k)},p(t))e^{\a w_k}).
\end{equation}
We provide an upper estimate of the right-hand side.
Let $G_R^{(k)}=\rho_R(F^{(k)},p(t))$ for $k\in\N$.
For the moment, we omit $p(t)$ from the notation and write, for example, $h_R(F^{(k)})$ instead of $h_R(F^{(k)},p(t))$.
We have
\begin{align*}
&-\cE\bigl(u^{(k)},h_R(F^{(k)})e^{\a w_k}\bigr)\\
&=-\cE\bigl(e^{-\a w_k}F^{(k)},e^{\a w_k}h_R(F^{(k)})\bigr)\\
&=-\cE^0\bigl(e^{-\a w_k}F^{(k)},e^{\a w_k}h_R(F^{(k)})\bigr)
-\int_E\left(b,\nb\bigl(e^{-\a w_k}F^{(k)}\bigr)\right)_He^{\a w_k}h_R(F^{(k)})\,d\mu\\*
&\qad-\int_E\left(c,\nb\bigl(e^{\a w_k}h_R(F^{(k)})\bigr)\right)_He^{-\a w_k}F^{(k)}\,d\mu
-\int_E V F^{(k)}h_R(F^{(k)})\,d\mu\\
&=-\frac12\int_E\Bigl(-\a e^{-\a w_k}F^{(k)}\nb w_k+e^{-a w_k}\nb F^{(k)},\\*
&\hspace{6em}\a e^{\a w_k}h_R(F^{(k)})\nb w_k+e^{\a w_k}g_R(F^{(k)})^2\nb  F^{(k)}\Bigr)_H\,d\mu\\*
&\qad-\int_E\left(b,-\a e^{-\a w_k}F^{(k)}\nb w_k+e^{-\a w_k}\nb F^{(k)}\right)_He^{\a w_k}h_R(F^{(k)})\,d\mu\\*
&\qad-\int_E\left(c,\a e^{\a w_k}h_R(F^{(k)})\nb w_k+ e^{\a w_k}g_R(F^{(k)})^2\nb F^{(k)}\right)_H e^{-\a w_k}F^{(k)}\,d\mu\\*
&\qad-\int_E V F^{(k)}h_R(F^{(k)})\,d\mu\\
&=-\frac12\int_E\Bigl\{g_R(F^{(k)})^2|\nb F^{(k)}|_H^2-\a^2 F^{(k)}h_R(F^{(k)})|\nb w_k|_H^2\\*
&\hspace{5em}+\a\bigl(h_R(F^{(k)})-F^{(k)}g_R(F^{(k)})^2\bigr)\left(\nb F^{(k)},\nb w_k\right)_H\Bigr\}d\mu\\*
&\qad+\a\int_E \left(b-c,\nb w_k\right)_H F^{(k)}h_R(F^{(k)})\,d\mu\\*
&\qad+\frac12\int_E \left(b-c,\nb F^{(k)}\right)_H \bigl(F^{(k)}g_R(F^{(k)})^2-h_R(F^{(k)})\bigr)\,d\mu\\*
&\qad-\frac12\int_E \left(b+c,\nb F^{(k)}\right)_H \bigl(F^{(k)}g_R(F^{(k)})^2+h_R(F^{(k)})\bigr)\,d\mu
-\int_E V F^{(k)}h_R(F^{(k)})\,d\mu\\
&=: I_1+I_2+I_3+I_4+I_5.
\end{align*}
Using \eqref{eq:dl1}, we have
\begin{align*}
I_1
&\le -\frac12\int_E g_R(F^{(k)})^2|\nb F^{(k)}|_H^2\,d\mu\\
&\qad+\frac{\a^2(1+\eps)}2  \int_E \rho_R(F^{(k)})^2|\nb w_k|_H^2\,d\mu
\qquad\text{(from Lemma~\ref{lem:6}(iii))}\\
&\qad+\frac{|\a|\eps}{2} \int_E \rho_R(F^{(k)})g_R(F^{(k)})\left|\left(\nb F^{(k)},\nb w_k\right)_H\right| d\mu\\
&\hspace{15em}\text{(from Lemma~\ref{lem:pos}(i) and Lemma~\ref{lem:6}(iv))}\\
&\le -\frac12\int_E \bigl|\nb G_R^{(k)}\bigr|_H^2\,d\mu
+\frac{\a^2(1+\eps)}{2}\int_E \bigl(G_R^{(k)}\bigr)^2\,d\mu
+\frac{|\a|\eps}{2}\int_E G_R^{(k)}\bigl|\nb G_R^{(k)}\bigr|_H d\mu\\
&\le -\cE^0\bigl( G_R^{(k)}\bigr)
+\frac{\a^2(1+\eps)}{2}\bigl\|G_R^{(k)}\bigr\|_2^2
+\eps\bigl\|\nb G_R^{(k)}\bigr\|_2^2+\frac{\eps \a^2}{16}\bigl\| G_R^{(k)}\bigr\|_2^2,\\
I_3
&=\int_E \left(b-c,\nb \sqrt{\iota_R(F^{(k)})}\right)_H  \sqrt{\iota_R(F^{(k)})}\,d\mu\\
&\le \kp\cE^0_1\Bigl(\sqrt{\iota_R(F^{(k)})}\Bigr)
\qquad\text{(from (B.1))}\\
&\le \kp K_0^2(p(t)-2)\left(\frac12\int_E |\nb F^{(k)}|_H^2 g_R(F^{(k)})^2\,d\mu 
+ \bigl\|G_R^{(k)}\bigr\|_2^2\right)
\qquad\text{(from Lemma~\ref{lem:small})}\\
&\le\eps\cE^0_1\bigl( G_R^{(k)}\bigr),\\
I_4+I_5
&=-\int_E \left(b+c,\nb (\psi_R(F^{(k)}))\right)_H  \psi_R(F^{(k)})\,d\mu
-\int_E V \psi_R(F^{(k)})^2\,d\mu\\
&\le\eta\cE^0\left(\psi_R(F^{(k)})\right)+\theta\left\| \psi_R(F^{(k)})\right\|_2^2
\qquad\text{(from (A.2))}\\
&\le \frac{\eta(1+2\dl)}{2}\int_E \bigl|\nb F^{(k)}\bigr|_H^2 g_R(F^{(k)})^2\,d\mu
+\theta(1+\dl)^2\bigl\| G_R^{(k)}\bigr\|_2^2\\*
&\hspace{19em}\text{(from Lemma~\ref{lem:6}(vi) and (iii))}\\
&\le (1+\eps)\eta\cE^0(G_R^{(k)})+(1+\eps)\theta\bigl\| G_R^{(k)}\bigr\|_2^2.
\end{align*}
Moreover, when $\a>0$,
\begin{align}
I_2
&=\a\int_E \left(b-c,\nb w\right)_H \psi_R(F^{(k)})^2\,d\mu\notag\\
&=\a\int_{\{0<\sd_B<\sd(A,B)\}} \left(b-c,\nb \sd_B\right)_H \psi_R(F^{(k)})^2\,d\mu
\quad\text{(from Proposition~\ref{prop:D})}\label{eq:I2}\\
&\le\a\Biggl\{\frac{\eps}{\a}\cE^0\bigl( \psi_R(F^{(k)})\bigr)+\lm_{\eps/\a}\bigl\| \psi_R(F^{(k)})\bigr\|_2^2\notag\\*
&\qad+\gm\bigl\| \psi_R(F^{(k)})\bigr\|_2\left(\int_E \psi_R(F^{(k)})^2 \log^+\frac{\psi_R(F^{(k)})^2}{\bigl\| \psi_R(F^{(k)})\bigr\|_2^2}\,d\mu\right)^{1/2}\Biggr\}\notag\\*
&\hspace{15em}\text{(from (B.2)$_{A,B}$ with $f=\psi_R(F^{(k)})/\|\psi_R(F^{(k)})\|_2$)}\notag\\
&\le\eps(1+\eps)\cE^0\bigl( G_R^{(k)}\bigr)+\a\lm_{\eps/\a}(1+\eps)\bigl\| G_R^{(k)}\bigr\|_2^2\notag\\*
&\qad+\eps\a^2(1+\eps)\bigl\| G_R^{(k)}\bigr\|_2^2+ \frac{\gm^2}{4\eps}\int_E \psi_R(F^{(k)})^2 \log^+\frac{\psi_R(F^{(k)})^2}{\bigl\| \psi_R(F^{(k)})\bigr\|_2^2}\,d\mu.\notag\\*
&\hspace{19em}\text{(from Lemma~\ref{lem:6}(vi) and (iii))}\notag
\end{align}
Assume $\a>0$ in what follows.
Combining \eqref{eq:lem7}, \eqref{eq:st1}, and the estimates from $I_1$ to $I_5$ above, we have
\begin{align*}
\sg'(t)
&=\lim_{R\to\infty}\lim_{k\to\infty}\left(-\cE\bigl(u^{(k)},h_R(F^{(k)},p(t))e^{\a w_k}\bigr)
-S\int_E {\del_y\tau}(F(t),p(t))\,d\mu\right)\\
&\le\varliminf_{R\to\infty}\varliminf_{k\to\infty}\biggl[C_0\cE^0\bigl( G_R^{(k)}\bigr)+C_1\bigl\|G_R^{(k)}\bigr\|_2^2
+\frac{\gm^2}{4\eps}\int_E \psi_R(F^{(k)})^2\log^+\frac{\psi_R(F^{(k)})^2}{\bigl\| \psi_R(F^{(k)})\bigr\|_2^2}\,d\mu\\*
&\hspace{6em}-S\int_E \bigl\{F(t)^{p(t)} \log^+ F(t)-2(\log 2)F(t)^2\bigr\}\,d\mu\biggr],
\quad\text{(from Lemma~\ref{lem:6''})}
\end{align*}
where
\begin{align}
C_0&=-1+3\eps+(1+\eps)\eta+\eps(1+\eps)\le0
\quad\text{(by the choice of $\eps$)}\notag\\
\shortintertext{and}
C_1&=\frac{(1+\eps)\a^2}{2}+\frac{\eps\a^2}{16}+\eps+(1+\eps)\theta+(1+\eps)\a\lm_{\eps/\a}+\eps(1+\eps)\a^2.
\label{eq:C1}
\end{align}
Let $G(t)=\sqrt{\tau(F(t),p(t))}$.
Note that $\|G(t)\|_2^2=\sg(t)$.
From Lemmas~\ref{lem:6}(iii)(v) and \ref{lem:4}(ii),
\begin{align}\label{eq:G1}
&0\le \psi_R(F^{(k)})\le \sqrt{1+\eps}G_R^{(k)}\le \sqrt{2}(1+\eps)G(t)\le2G(t)\quad\text{$\mu$-a.e.}\notag\\*
&\text{for all $k\in\N$ and $R>2$}
\end{align}
and
\begin{equation}\label{eq:G2}
F(t)\le G(t)\le F(t)\vee F(t)^{p(t)/2}.
\end{equation}
Since $\log^+(x/y)\le\log^+x+\log^- y$ for $x,y>0$ and the maps $[0,\infty)\ni a\mapsto a\log^+ a\in[0,\infty)$ and $[0,\infty)\ni a\mapsto a+a\log^- a\in[0,\infty)$ are both non-decreasing, we obtain, for every $k\in\N$ and $R>2$, that
\begin{align*}
&\int_E \psi_R(F^{(k)})^2\log^+\frac{\psi_R(F^{(k)})^2}{\bigl\| \psi_R(F^{(k)})\bigr\|_2^2}\,d\mu\\
&\le \int_E \psi_R(F^{(k)})^2\log^+\frac{\psi_R(F^{(k)})^2}{4}\,d\mu
+\bigl\|\psi_R(F^{(k)})\bigr\|_2^2\log^-\frac{\bigl\| \psi_R(F^{(k)})\bigr\|_2^2}{4}\\
&\le 4\int_E G(t)^2 \log^+ G(t)^2\,d\mu+4\left(\|G(t)\|_2^2+\|G(t)\|_2^2 \log^-\|G(t)\|_2^2\right)
\quad\text{(from \eqref{eq:G1})}
\end{align*}
and
\begin{align*}
\int_E G(t)^2 \log^+ G(t)^2\,d\mu
&\le \int_E\bigl(F(t)^2\vee F(t)^{p(t)}\bigr)\log^+\bigl(F(t)^2\vee F(t)^{p(t)}\bigr)\,d\mu
\quad\text{(from \eqref{eq:G2})}\\
&= p(t)\int_E F(t)^{p(t)}\log^+ F(t)\,d\mu.
\end{align*}
Here, in the last equality, we used the identity
\[
(a^2\vee a^{p(t)})\log^+(a^2 \vee a^{p(t)})
= p(t) a^{p(t)}\log^+ a
\quad\text{for }a\ge0.
\]
Because we set $S=3\gm^2/\eps$,
\begin{align}
\sg'(t)
&\le (2+2\eps)C_1\|G(t)\|_2^2
+\frac{p(t)S}{3}\int_E F(t)^{p(t)}\log^+ F(t)\,d\mu\notag\\
&\qad+\frac{S}{3}\bigl(\|G(t)\|_2^2+\|G(t)\|_2^2\log^-\|G(t)\|_2^2\bigr)\notag\\
&\qad-S\int_E F(t)^{p(t)} \log^+ F(t)\,d\mu
+2S(\log 2)\|F(t)\|_2^2\notag\\
&\le U\sg(t)+W \sg(t)\log^- \sg(t),\label{eq:diffeq}
\end{align}
where
\begin{equation}\label{eq:UV}
U=(2+2\eps)C_1+\Bigl(\frac13+2\log 2\Bigr)S
\quad\text{and}\quad
W=\frac{S}{3}.
\end{equation}
\subsection{Solving the differential inequality}
We give an explicit upper bound of $\sg(t)$.
Since $\|u_0\|_2>0$, there exists some $t_1\in(0,t_0]$ such that $\|u_t\|_2>0$ for $t\in[0,t_1]$.
For this step, we consider only $t\in[0, t_1]$.
Keeping in mind that
\[
(\log\sg)'(t)=\sg'(t)/\sg(t)\le U+W\log^-\sg(t)
\]
from \eqref{eq:diffeq}, we define
\begin{equation}\label{eq:chi}
\chi(x)=\int_0^x\frac1{U+W\max\{-s,0\}}\,ds
=\begin{cases}
-\dfrac1W\log\left(1-\dfrac{W}{U}x\right)&(x\le0),\smallskip\\
\dfrac{x}{U} &(x>0).
\end{cases}
\end{equation}
Then, 
\[
\bigl(\chi(\log \sg)\bigr)'(t)=\chi'(\log \sg(t))\bigl(\log \sg\bigr)'(t)\le1,
\]
which implies
\[
\chi(\log \sg(t))\le \chi(\log \sg(0))+t,\qquad t\ge0.
\]
This inequality implies
\begin{equation}\label{eq:logsg}
\log \sg(t)\le \begin{cases}\dfrac{U}{W}\left[1-\exp\bigl(-W\{\chi(\log \sg(0))+t\}\bigr)\right] & \text{if }\sg(t)\le1,\\
U\{\chi(\log \sg(0))+t\} & \text{if }\sg(t)>1.
\end{cases}
\end{equation}
We can confirm that 
\[
\dfrac{U}{W}\left[1-\exp(-W z)\right]\le U z,\quad z\in\R,
\]
so that \eqref{eq:logsg} implies
\begin{equation}\label{eq:sg}
\sg(t)\le\exp\bigl(U \{\chi(\log \sg(0))+ t\}\bigr).
\end{equation}

For the proof of Theorem~\ref{th:upper}, we assume that $\sd(A,B)>0$ because otherwise the assertion is trivial.
Define
\begin{equation}\label{eq:sg12}
\sg_1(t,\a)=\int_E \tau(e^{\a w}T_t\bone_B,p(t))\,d\mu,\quad
\sg_2(t,\a)=\int_E \tau(e^{-\a w}\hat T_t\bone_A,p(t))\,d\mu
\end{equation}
for $t>0$ and $\a>0$.
Both $\sg_1(t,\a)$ and $\sg_2(t,\a)$ have the same kind of estimates as \eqref{eq:sg}.
Indeed, for the estimate of $\sg_2$, the discussion in the previous subsection is applied with $b$, $c$, and $\a$  replaced by $c$, $b$, and $-\a$, respectively.
The only term that requires care is $I_2$, but the estimate \eqref{eq:I2} is unchanged by this replacement.
From the Cauchy--Schwarz inequality and Lemma~\ref{lem:4}(ii),
\begin{align}
P_t(A,B)
&\le\left\{\int_E(e^{\a w}T_{t/2}\bone_B)^2\,d\mu\right\}^{1/2}\left\{\int_E(e^{-\a w}\hat T_{t/2}\bone_A)^2\,d\mu\right\}^{1/2}\notag\\
&\le \sg_1(t/2,\a)^{1/2}\sg_2(t/2,\a)^{1/2}.
\label{eq:PtAB}
\end{align}
Letting $N=\sd(A,B)$ and $\a=N/t$, we have
\begin{equation}\label{eq:upperC}
\varlimsup_{t\to0}t\log P_t(A,B)
\le \varlimsup_{t\to0}\frac{t}{2}\log \sg_1(t/2,N/t)
+ \varlimsup_{t\to0}\frac{t}{2}\log \sg_2(t/2,N/t).
\end{equation}
We also have
\begin{align*}
\sg_1(0,N/t)&=\int_E \tau(\bone_B,q)\,d\mu=\mu(B),\\
\sg_2(0,N/t)&=\int_E \tau(e^{-N^2/t}\bone_A,q)\,d\mu=\mu(A)e^{-2N^2/t},\\
\frac{t}{2}\log\sg_j\Bigl(\frac{t}{2},\frac{N}{t}\Bigr)
&=\frac{U t}{2}\chi\Bigl(\log \sg_j\Bigl(0,\frac{N}{t}\Bigr)\Bigr)+\frac{Ut^2}{4},
\quad j=1,2.
\end{align*}
We remark that $U$ and $\chi$ depend on $\a$
(see \eqref{eq:UV} and \eqref{eq:C1}).
When $\a=N/t$,
\begin{align*}
\lim_{t\to0}U t^2 
&=(2+2\eps)\left(\frac{1+\eps}{2}+\frac{\eps }{16}+\eps \right)N^2
\qquad\text{(from \eqref{eq:lmeps})}\\
&=: \b(\eps)N^2
\end{align*}
 and 
\[
\lim_{t\to0}U\chi(x)=x \quad\text{for }x\in\R,
\]
in view of \eqref{eq:chi}.
In particular, $U=O(t^{-2})$ as $t\to0$.
We also remark that $\lim_{\eps\to0}\b(\eps)=1$.
Then, we obtain
\[
\frac{U t}{2}\chi\Bigl(\log \sg_1\Bigl(0,\frac{N}{t}\Bigr)\Bigr)+\frac{U t^2}{4}
=\frac{t}{2}U\chi(\log \mu(B))+\frac{U t^2}{4}
\to \frac{\b(\eps)N^2}{4} \qquad\text{as }t\to0.
\]
Therefore,
\begin{equation}\label{eq:upperA}
\varlimsup_{t\to0}\frac{t}{2}\log \sg_1(t/2,N/t)\le \frac{\b(\eps)N^2}{4}.
\end{equation}
Also, for $t$ small enough that $\sg_2(0,N/t)<1$,
\begin{align*}
\frac{Ut}{2}\chi\Bigl(\log \sg_2\Bigl(0,\frac{N}{t}\Bigr)\Bigr)+\frac{U^2 t}{4}
&=\frac{t}{2}\left(-\frac{U}{W}\log\left(1-\frac{W}{U}\Bigl(\log \mu(A)-\frac{2N^2}{t}\Bigr)\right)+\frac{U t}{2}\right)\\
&=\frac{U t}{2W}\left(-\frac{2W N^2}{Ut}+O(t^2)\right)+\frac{U t^2}{4}\\
&\to -N^2+\frac{\b(\eps)}{4}N^2 \qquad\text{as }t\to0.
\end{align*}
Therefore,
\begin{equation}\label{eq:upperB}
\varlimsup_{t\to0}\frac{t}{2}\log \sg_2(t/2,N/t)
\le \left(-1+\frac{\b(\eps)}{4}\right)N^2.
\end{equation}
By combining \eqref{eq:upperC}, \eqref{eq:upperA}, and \eqref{eq:upperB}, 
\[
\varlimsup_{t\to0}t\log P_t(A,B)\le \left(-1+\frac{\b(\eps)}{2}\right)N^2.
\]
Letting $\eps\to0$, we obtain \eqref{eq:upper},
which finishes the proof of Theorem~\ref{th:upper}.
\begin{rem}
\begin{enumerate}
\item As seen from the proof, when we can let $\gm=0$ in (B.2)$_{A,B}$, the $L^p$-analysis is not necessary and the proof becomes much simpler.
\item
If $(E,\cB,\mu)$ is a finite measure space, 
then we can define $\sg$, $\sg_1$, and $\sg_2$ as 
\begin{align*}
\sg(t)&=\int_E F(t)^{p(t)}\,d\mu,\\
\sg_1(t,\a)&=\int_E (e^{\a w}T_t\bone_B)^{p(t)}\,d\mu,\\
\sg_2(t,\a)&=\int_E (e^{-\a w}\hat T_t\bone_A)^{p(t)}\,d\mu
\end{align*}
and use the inequality
\[
P_t(A,B)\le \mu(E)^{1-\frac2{p(t/2)}}\sg_1(t/2,\a)^{\frac1{p(t/2)}}\sg_2(t/2,\a)^{\frac1{p(t/2)}}
\]
in place of \eqref{eq:sigma}, \eqref{eq:sg12}, and \eqref{eq:PtAB}.
This change makes the proof of Theorem~\ref{th:upper} shorter and simpler since the fine estimates in Section~3.2 are not necessary.
\end{enumerate}
\end{rem}
\section{Proof of Theorem~\protect{\ref{th:lower}}}
\subsection{Cutoff functions and their properties}
We turn to the lower-side estimate
and prove Theorem~\ref{th:lower}. 
In Section~2.1 of \cite{HR}, some nice concave functions are introduced
as cutoff functions. 
Because our semigroup $\{T_t\}_{t>0}$ does not have the Markov property in general, we need to modify these functions to be suitable.
First, we take a real-valued function $g$ on $\R$ satisfying the following properties:
\begin{itemize}
\item $g$ is an odd and bounded $C^3$-function;
\item $g(x)=x$ for $x \in[-1,1]$ and $0<g'(x) \le 1$ on $\R$; and
\item there is a positive constant $C$ such that $0 \le -g''(x) \le Cg'(x)$ for $x\in[-1,\infty)$. 
\end{itemize}
These conditions imply that $\lim_{x \to \infty} g(x)=L$, $\lim_{x \to -\infty}g(x)=-L$ for some $L>1$ and that the convergence is monotone. Note that $g$ is concave on $[-1,\infty)$.

Define our main cutoff functions at level $K>0$ by 
\begin{equation*}
 \phi^{K}(x)=Kg(x/K),\ 
 \Phi^K(x)=\int_{0}^{x}(\phi^K)'(s)^{2}\,ds,\text{ and }\Psi^K(x)=x (\phi^K)'(x)^{2}.
\end{equation*}
From the conditions on $g$, we have the following properties:
\begin{enumerate}
\item[(C.1)]  $0<(\phi^K)'(x)\le1$, 
\item[(C.2)]  
$ 0\le |(\phi^K)''(x)|\le CK^{-1}(\phi^K)'(x)$,
\item[(C.3)] 
$  0 \le \Psi^K(x) \le \Phi^K(x) \le \phi^K(x) \le LK$ for $x \ge 0$, 
and 
$|\Psi^K(x)|, |\Phi^K(x)|, |\phi^K(x)| \le LK $ on $\R$,
\item[(C.4)] $\phi^K(x)=\Phi^K(x)=\Psi^K(x)=x $ on $[-K,K]$,
\item[(C.5)] 
$\lim_{x \to \infty}\Psi^K(x)=0$,
\item[(C.6)] $ \Phi^K(x/\b) \ge \Phi^K(x)/\b$ for all $\b>1$ and $x \ge 0$.
\end{enumerate}
To simplify the notation,  
we omit explicit indication of the dependency on $K$ for most of this section. For example, we write $\phi$ instead of $\phi^{K}$ whenever the value of $K$ is clear from the context. 
The monotonicities of $\phi$, $\Phi$, (C.3), and (C.5) guarantee that 
$\phi$, $\Phi$, and $\Psi$ can be extended to continuous functions on $[-\infty,\infty]$,
and these extensions use the same symbols.
The following estimates result:
\begin{enumerate}
\item[(C.7)] $ \Phi^{K}(x)-\Phi^{K}(\Phi^{M}(x)) \le \Phi^{K}(\infty)-\Phi^{K}(M)$ for all $K,M>0$ and $x \in \R$.
\end{enumerate}
Indeed, this inequality is trivial when 
$x\le M$,
and when $x>M$ it is deduced from $\Phi^{K}(x)\le \Phi^{K}(\infty)$ and $\Phi^{K}\bigl(\Phi^{M}(x)\bigr) \ge \Phi^{K}(M)$.

We also introduce a function $\hat\Phi^K$ on $\R$, defining it as
\begin{equation}\label{eq:Khat}
\hat\Phi^K(x)=\begin{cases}
\Phi^K(x) & (x\ge -K),\\
x & (x<-K).\end{cases}
\end{equation}
This function is concave and 1-Lipschitz on $\R$.

For functions
 $u_{t}^{\dl}$ on $E$ with parameters $t$ and $\dl$, we write $\phi_{t}^{\dl}$ for $\phi(u_{t}^{\dl})$, $\Phi_{t}^{\dl}$ for $\Phi(u_{t}^{\dl})$, $\bar{\phi}_{t}^{\dl}$ for $t^{-1}\int_{0}^{t}\phi_{s}^{\dl}\,ds$, and so on. We denote  $\int_{E}fg\,d\mu$ by $(f,g)_\mu$ for functions $f$ and $g$ on $E$.
 
  For $\dl \in \left(0,1 \right]$ and $f \in L^{2}(\mu)$ with $f\ge0$ $\mu$-a.e.\ and $t>0$, we define
\[
u_{t}^{\dl}(x)=-t \log \left(T_{t}f(x)+\dl \right), \quad e_{t}^{\dl}=-t \log \dl{}.
\]
We need the following lemmas, introduced in \cite{HR}.
\begin{lemma}
\label{lem:convexlem} 
The function $F(x)=\Phi(-t \log x)$ is convex for $x \in [0,\infty]$ 
if $0<t\le K/(2C)$.
\end{lemma}
\begin{proof} 
The proof here follows the proof of Lemma 2.1 in \cite{HR}.
Compute
\begin{align*} 
F'(x)&=-\frac{t}{x} \Phi'(-t \log x), \\
F''(x)&=\frac{t}{x^{2}}\Phi'(-t \log x)+\frac{t^{2}}{x^{2}}\Phi''(-t \log x)\\
&=\frac{t}{x^2}\left.\left((\ph')^2+2t\ph'\ph''\right)\right|_{-t\log x}.
\end{align*}
From this and (C.2), $F$ has a nonnegative second derivative.
\end{proof}
\begin{lemma}[{%
\cite[Lemma 2.2]{HR}}]\label{lem:lsclem}
Suppose that $F$ is a concave continuous function defined on $\R$. If $f_{n} \to f$ weakly in $L^{2}(\mu)$, $F(f_{n}) \in L^{2}(\mu)$ for each $n$, and $F(f_{n})$ has a subsequence that converges to some function 
$\hat{F}$ weakly in $L^{2}(\mu)$, then $\hat{F} \le F(f)$.
\end{lemma}
\begin{lemma}[{%
\cite[Lemma 2.4]{HR}}]\label{lem:wklem}
Let $\{f_n\}$ be a sequence of $\D$ that converges weakly to some $f$ in $\D$.
Then,
\[
\varliminf_{n\to\infty}\int_E |\nb f_n|_H^2 h\,d\mu \le \int_E |\nb f|_H^2 h\,d\mu
\]
 for $h\in \D$ with $0\le h\le M$ $\mu$-a.e.\ for some $M\ge0$.
\end{lemma}
\begin{lemma}[{cf.\ \cite[Lemma 2.5]{HR}}]\label{lem:meanlem} 
Let $T>0$ and suppose
that $f(t,x)=f_{t}(x)$ is a bounded jointly measurable function for $(t,x)\in \left(0,T \right] \times E$. Also suppose that $f_{t} \in\D$ for each $t \in \left(0,T\right]$ and  that
$ \int_{0}^{T} \cE^0(f_t)\,dt<\infty$.
Writing 
\[
\bar{f}_{T}=\frac{1}{T} \int_{0}^{T}f_{t}\,dt,
\]
we have $\bar{f}_{T} \in\D$, and the following is true for any nonnegative $h\in \D_{b}$:
\[
\int_{E}  |\nb \bar{f}_{T}|_{H}^2 h\,d\mu \le \frac{1}{T} \int_{0}^{T} \int_{E} |\nb f_{t}|_{H}^2 h\,d\mu \,dt.
\]
\end{lemma}
\subsection{Rough estimates}
In the following, $\Gm(f,g)$ and $\Gm(f)$ denote $(\nb f,\nb g)_{H}$ and $(\nb f, \nb f)_{H}$, respectively. 
\begin{lemma}\label{lem:lem3.2.1}
$u_{t}^{\dl}-e_{t}^{\dl} \in \D$ and 
\begin{align}
(\rho, \del_{t}\Phi_{t}^{\dl} )_\mu&=\frac{1}{t}(\rho,\Psi_{t}^{\dl} )_\mu- \cE^{0} \bigl( ((\phi')_{t}^{\dl})^{2}\rho, u_{t}^{\dl}-e_{t}^{\dl}\bigr) 
-\frac{1}{2t} \int_E \Gm ( u_{t}^{\dl}-e_{t}^{\dl})((\phi')_{t}^{\dl})^{2}\rho\,d\mu \nonumber\\
&\qad -\int_E \bigl(b,\nb (u_{t}^{\dl}-e_{t}^{\dl}  ) \bigr)_{H}((\phi')_{t}^{\dl})^{2}\rho\,d\mu 
+t \int_E (c, \nb  \rho )_{H} \frac{ ((\phi')_{t}^{\dl})^{2}  T_{t}f}{T_{t}f+\dl}\,d\mu \notag \\	
&\qad + \int_E \left(c, \nb \bigl(u_t^{\delta}-e_t^{\delta} \bigr)\right)_{H} \left\{ 2t (\phi')_t^\delta (\phi'')_t^\delta +((\phi')_{t}^{\dl})^{2} \right\} \frac{ \rho T_{t}f}{T_{t}f+\dl}\,d\mu\notag\\
&\qad + t\int_E V\frac{((\phi')_{t}^{\dl})^{2}\rho T_t f}{T_t f+\dl}\,d\mu \label{identitypde}
\end{align}
for $\rho \in \bigcup_{k=1}^\infty\D_{E_k,b}$.
\end{lemma}
\begin{proof}
Let $\Xi(s)=\log (s+\dl)-\log \dl$ for $s\in[0,\infty)$.
From $\Xi(0)=0$, the boundedness of the derivative of $\Xi$ on $[0,\infty)$, and
$u_{t}^{\dl}-e_{t}^{\dl}=-t \Xi(T_{t}f)$, we conclude that $u_{t}^{\dl}-e_{t}^{\dl} \in \D$.

We prove the identity \eqref{identitypde}. 
First, 
\begin{align}
(\rho, \del_{t}u_{t}^{\dl})_\mu= \frac{1}{t}(\rho,u_{t}^{\dl})_\mu-\left(\rho, \frac{t \cL T_{t}f}{T_{t}f+\dl} \right)_\mu
=\frac{1}{t}(\rho,u_{t}^{\dl})_\mu+t \cE \left(T_{t}f,\frac{\rho}{T_{t}f+\dl}  \right)\label{eq:pde1}
\end{align}
and
\begin{align}
\cE \bigl(\rho, \Xi(T_{t}f)\bigr) 
&=\frac{1}{2} \int_E \Gm \bigl(\rho, \Xi(T_{t}f) \bigr)\,d\mu +\int_E (b,\nb \rho)_{H} \Xi(T_{t}f) \,d\mu 
+ \int_E \bigl(c, \nb (\Xi(T_{t}f) )  \bigr)_{H} \rho \,d\mu\notag\\
&\qad+\int_E V\rho\Xi(T_t f)\,d\mu. 
\label{eq:pde2}
\end{align}
The first term on the right-hand side is computed as 
\begin{align}
&\frac{1}{2} \int_E \Gm(\rho, \Xi(T_{t}f))\,d\mu\notag\\
&=\frac{1}{2} \int_E \Gm(\rho, T_{t}f)\frac{1}{T_{t}f+\dl} \,d\mu \notag  \\  
&=\frac{1}{2} \int_E \left\{ \Gm \left(\frac{\rho}{T_{t}f+\dl},T_{t}f  \right)+\Gm(T_{t}f) \frac{\rho}{(T_{t}f+\dl)^{2}} \right\} \,d\mu  \notag  \\
&=\cE \left(T_{t}f, \frac{\rho}{T_{t}f+\dl}\right)-\int_E \left(b,\nb T_{t}f\right)_{H} \frac{\rho}{T_{t}f+\dl} \,d\mu 
-\int_E \left(c,\nb \left(\frac{\rho}{T_{t}f+\dl} \right) \right)_{H} T_{t}f \,d\mu\notag\\  
&\qad-\int_E V\frac{\rho T_t f}{T_t f+\dl}\,d\mu
+\frac{1}{2} \int_E \Gm( \Xi(T_{t}f))\rho \,d\mu. 
\label{eq:pde3}
\end{align}
Combining the identities \eqref{eq:pde1}, \eqref{eq:pde2}, and \eqref{eq:pde3}, it holds that
\begin{align*}
(\rho, \del_{t}u_{t}^{\dl})_\mu
&=\frac{1}{t}(\rho,u_{t}^{\dl})_\mu+t \cE \bigl(\rho, \Xi(T_{t}f)\bigr)
+t\int_E (b,\nb T_{t}f)_{H}\frac{\rho}{T_{t}f+\dl} \,d\mu\\
&\qad+t\int_E \left(c,\nb \left(\frac{\rho}{T_{t}f+\dl} \right) \right)_{H}T_{t}f\,d\mu
+t\int_E V\frac{\rho T_t f}{T_t f+\dl}\,d\mu
-\frac{t}{2} \int_E \Gm \bigl( \Xi(T_{t}f) \bigr)\rho \,d\mu\\
&\qad -t\int_E (b,\nb \rho)_{H}  \Xi(T_{t}f) \,d\mu
-t\int_E \bigl(c,\nb ( \Xi(T_{t}f) )\bigr)_{H} \rho \,d\mu 
-t\int_E V \rho \Xi(T_tf)\,d\mu\\
&= \frac{1}{t}(\rho,u_{t}^{\dl})_\mu- \cE^{0} (\rho, u_{t}^{\dl}-e_{t}^{\dl}) -\frac{1}{2t} \int_E \Gm( u_{t}^{\dl}-e_{t}^{\dl})\rho \,d\mu \\
&\qad +t\int_E (b,\nb T_{t}f )_{H}\frac{\rho}{T_{t}f+\dl} \,d\mu
 +t\int_E \left(c,\nb \left(\frac{\rho}{T_{t}f+\dl} \right) \right)_{H}T_{t}f\,d\mu\notag\\
&\qad +t\int_E V\frac{\rho T_t f}{T_t f+\dl}\,d\mu.
\end{align*}
By using the identity 
$(\rho, \del_{t} \Phi_{t}^{\dl})_\mu=\bigl(((\phi')_{t}^{\dl})^{2}\rho,\del_{t}u_{t}^{\dl} \bigr)_\mu$ and replacing $\rho$  with $((\phi')_{t}^{\dl})^{2}\rho$ in the relation above, we obtain 
\begin{align}
\del_{t}(\rho, \Phi_{t}^{\dl})_\mu
&=\frac{1}{t}\bigl(((\phi')_{t}^{\dl})^{2}\rho,u_{t}^{\dl}\bigr)_\mu- \cE^{0} \bigl(  ((\phi')_{t}^{\dl})^{2}\rho, u_{t}^{\dl}-e_{t}^{\dl}\bigr) 
-\frac{1}{2t} \int_E \Gm( u_{t}^{\dl}-e_{t}^{\dl})((\phi')_{t}^{\dl})^{2}\rho \,d\mu \notag \\
&\qad +t\int_E (b,\nb T_{t}f )_{H}\frac{((\phi')_{t}^{\dl})^{2}\rho}{T_{t}f+\dl} \,d\mu 
 +t\int_E \left(c,\nb \left( \frac{((\phi')_{t}^{\dl})^{2}\rho}{T_{t}f+\dl} \right) \right)_{H} T_{t}f \,d\mu\notag\\
 &\qad+t\int_E V\frac{((\phi')_{t}^{\dl})^{2}\rho T_t f}{T_t f+\dl}\,d\mu.
\label{eq:pde4}
\end{align}
Here, we have
\begin{equation}\label{eq:pde6}
t\int_E (b,\nb T_{t}f )_{H}\frac{((\phi')_{t}^{\dl})^{2}\rho}{T_{t}f+\dl} \,d\mu 
=-\int_E \bigl(b,\nb (u_{t}^{\dl}-e_{t}^{\dl})  \bigr)_{H}((\phi')_{t}^{\dl})^{2}\rho \,d\mu
\end{equation}
and
\begin{align}
&t\int_E \left(c,\nb \left( \frac{((\phi')_{t}^{\dl})^{2}\rho}{T_{t}f+\dl} \right) \right)_{H} T_{t}f \,d\mu \notag \\ 
&= t \int_E \bigl(c, \nb (((\phi')_{t}^{\dl})^{2}\rho) \bigr)_{H} \frac{T_{t}f}{T_{t}f+\dl} \,d\mu  + \int_E \bigl(c , \nb (u_{t}^{\dl}-e_{t}^{\dl}) \bigr)_{H} \frac{ ((\phi')_{t}^{\dl})^{2} \rho T_{t}f}{T_{t}f+\dl }\, d\mu \notag \\
&=t \int_E (c, \nb  \rho )_{H} \frac{ ((\phi')_{t}^{\dl})^{2}  T_{t}f}{T_{t}f+\dl}\,d\mu + \int_E \left(c, \nb \bigl(u_t^{\dl}-e_t^{\dl} \bigr)\right)_{H} \left\{ 2t (\phi')_t^\dl (\phi'')_t^\dl +((\phi')_{t}^{\dl})^{2} \right\} \frac{ \rho T_{t}f}{T_{t}f+\dl}\,d\mu.
\label{eq:pde5}
\end{align}
Substituting \eqref{eq:pde6} and \eqref{eq:pde5} for \eqref{eq:pde4}, we obtain \eqref{identitypde}.
\end{proof}
\begin{lemma}\label{lem:R}
For $k\in\N$ and $\rho\in \D_{E_k,b}$ 
with $0\le \rho\le 1$ $\mu$-a.e.,
\begin{align*}
&\left|(\rho, \del_{t}\Phi_{t}^{\dl} )_\mu-\frac{1}{t}(\rho,\Psi_{t}^{\dl} )_\mu+ \cE^{0} \bigl( ((\phi')_{t}^{\dl})^{2}\rho, u_{t}^{\dl}-e_{t}^{\dl}\bigr) 
+\frac{1}{2t} \int_E \Gm ( u_{t}^{\dl}-e_{t}^{\dl})((\phi')_{t}^{\dl})^{2}\rho\,d\mu\right|\\
&\le \left(\frac{C t}{K}+1\right)\int_E \Gm ( u_{t}^{\dl}-e_{t}^{\dl})((\phi')_{t}^{\dl})^{2}\rho\,d\mu+t\cE^0(\rho)+R(t,k),
\end{align*}
where we define
\[
R(t,k):= \frac12\int_{E_k}|b|_H^2\,d\mu+\left(\left(\frac{C}{K}+\frac12\right)t+\frac12\right)\int_{E_k}|c|_H^2\,d\mu+t\int_{E_k}|V|\,d\mu.
\]
\end{lemma}
\begin{proof}
Using Proposition~\ref{prop:D}, (C.1), and (C.2), we have
\begin{align*}
&\left| \int_E \bigl(b,\nb (u_{t}^{\dl}-e_{t}^{\dl} ) \bigr)_{H}\bigl((\phi')_{t}^{\dl}\bigr)^{2}\rho\,d\mu \right| \le \frac{1}{2}\int_{E_k} |b|^{2}_{H} \,d\mu +\frac{1}{2}\int_{E}\Gm(u_{t}^{\dl}-e_{t}^{\dl}) ((\phi')_t^\dl)^2 \rho \,d\mu
,\\ 
&  \left| \int_E (c, \nb \rho )_{H} \frac{ ((\phi')_{t}^{\dl})^{2}  T_{t}f}{T_{t}f+\dl}\,d\mu \right|  \le \frac{1}{2}\int_{E_k}  |c|^2_{H} \,d\mu+\cE^{0}(\rho), \\
& \left| \int_E \left(c, \nb \bigl(u_t^{\delta}-e_t^{\delta} \bigr)\right)_{H} \left\{ 2t (\phi')_t^\delta (\phi'')_t^\delta +((\phi')_{t}^{\dl})^{2} \right\} \frac{ \rho T_{t}f}{T_{t}f+\dl}\,d\mu \right| \\
&\quad \le \left( \frac{Ct}{K}+\frac{1}{2} \right) \int_{E_k}|c|^2_{H}\,d\mu+\left(\frac{Ct}{K}+\frac{1}{2}\right)\int_{E}\Gm(u_{t}^{\dl}-e_{t}^{\dl}) ((\phi')_t^\dl)^2 \rho \,d \mu, \\
&  \left| \int_E \frac{((\phi')_{t}^{\dl})^{2}\rho\,T_{t}f}{T_{t}f+\dl}\,Vd\mu \right| \le \int_{E_k} |V|\,d\mu.
\end{align*}
These estimates and Lemma~\ref{lem:lem3.2.1} together imply the claim.
\end{proof}
\begin{prop}
There exists a measurable nest $\{\hat E_k\}_{k=1}^\infty$ and functions $\{\chi_k\}_{k=1}^\infty$ in $\D$ such that, for every $k$, $\hat E_k\subset E_k$, $0\le \chi_k\le 1$ on $E$, $\chi_k=1$ on $\hat E_k$, and $\chi_k=0$ on $E\setminus E_k$.
\end{prop}
\begin{proof}
For $k\in\N$, there exists an $h_k\in\D$ such that $h_k\ge1$ on $E_k$. 
Take a function $g_k$ from $\bigcup_{n=1}^\infty \D_{E_n}$ such that $\cE^0_1(h_k-g_k)\le 1/k$, and set $Y_k=\{x\in E_k\mid g_k(x)\ge1/2\}$.
Let $\rho_k$ denote $0\vee \bigl(2(h_k-g_k) \wg 1\bigr)$. Then, $\cE^0_1(\rho_k)\le 4/k$ and $\rho_k=1$ on $E_k\setminus Y_k$.
We prove that $\bigcup_{k=1}^\infty \D_{Y_k}$ is dense in $\D$.
Take $f\in \D_{E_n,b}$ for some $n\in\N$. For $k\ge n$, let $f_k=f-f \rho_k$.
Then, $f_k=0$ $\mu$-a.e.\ on $(E\setminus E_n)\cup (E_k\setminus Y_k)\supset E\setminus Y_k$. It follows that $f_k\in \D_{Y_k}$. Moreover, since
\[
\cE^0(f-f_k)^{1/2}=\cE^0(f \rho_k)^{1/2}\le \|f\|_\infty \cE^0(\rho_k)^{1/2}+\cE^0(f)^{1/2}
\le \|f\|_\infty(4/k)^{1/2}+\cE^0(f)^{1/2},
\]
which is bounded in $k$, and
\[
\|f-f_k\|_2\le \|f\|_\infty \|\rho_k\|_2 \le \|f\|_\infty(4/k)^{1/2}\to 0 \quad(k\to\infty),
\]
the Ces\`aro means of a certain subsequence of $\{f_k\}_{k=n}^\infty$ belonging to $\bigcup_{k=n}^\infty \D_{Y_k}$ converge to $f$ in $\D$.
Thus $\bigcup_{k=1}^\infty \D_{Y_k}$ is dense in $\D$ because $\bigcup_{n=1}^\infty \D_{E_n,b}$ is dense in $\D$.

Let $Z_n=\bigcup_{k=1}^n Y_k$ and define $\eta_n=0\vee\bigl(2(g_1\vee\dots\vee g_n)\wg1\bigr)$. Then, $Z_n\subset E_n$ and $\bigcup_{n=1}^\infty \D_{Z_n}$ is dense in $\D$. Moreover, $0\le \eta_n\le1$ on $E$, $\eta_n=1$ on $Z_n$ and $\eta_n\in\bigcup_{k=1}^\infty\D_{E_k}$.
Take a strictly increasing sequence $\{m(n)\}_{n=1}^\infty$ such that $\eta_n\in \D_{E_{m(n)}}$ for every $n\in\N$ and define
\begin{align*}
& \hat E_k=\emptyset,\ \chi_k=0\quad \text{for }1\le k< m(1),\\
& \hat E_k=Z_{n},\ \chi_k=\eta_n\quad \text{for }m(n)\le k<m(n+1),\ n=1,2,\dots. 
\end{align*}
Then, $\{\hat E_k\}_{k=1}^\infty$ and $\{\chi_k\}_{k=1}^\infty$ satisfy the required conditions.
\end{proof}
Define $u_{t}^{\dl}=-t \log \left( T_{t}\bone_{B}+\dl \right)$ for $B \in \cB_0$. 
\begin{lemma} \label{lem:bddlem}
There exists a positive constant $T_0$, depending only on $C$ and $K$, such that both $\left\{ \cE^{0}( \bar{\phi}_{t}^{\dl}\chi_k) \right\}_{ 0<t \le T_{0}, \,0< \dl \le 1 }$ and $\left\{ \cE^{0}(\bar{\Phi}_{t}^{\dl}\chi_k) \right\}_{ 0<t \le T_{0},\,0< \dl \le 1 }$ are bounded for each $k$.
\end{lemma}
\begin{proof}
Let $U_{t}^{\dl}=2 \cE^{0}(\phi_{t}^{\dl}\chi_k)$,
$W_{t}^{\dl}= \int_E\Gm({u_{t}^{\dl}-e_{t}^{\dl}})((\phi')_{t}^{\dl})^{2} \chi_k^{2}\,d\mu$,
 and $a_{k}=2\cE^{0}(\chi_k)$. 
Applying the chain rule~\eqref{eq:chain},
\begin{align}
U_{t}^{\dl}&=\int_E \Gm( \phi_{t}^{\dl}\chi_k, \phi_{t}^{\dl}\chi_k)\,d\mu \notag \\
&=\int_E\Gm({u_{t}^{\dl}-e_{t}^{\dl}}) ((\phi')_{t}^{\dl})^{2} \chi_k^{2}\,d\mu
+2\int_E\Gm({u_{t}^{\dl}-e_{t}^{\dl},\chi_k})\phi_{t}^{\dl}(\phi')_{t}^{\dl}\chi_k\,d\mu
+\int_E\Gm({\chi_k})(\phi_{t}^{\dl})^{2}\,d\mu \notag \\
&\le 2 W_{t}^{\dl}+2K^{2}L^{2}a_{k}\qquad\text{(from (C.3))}. 
\label{eq:eq20}
\end{align}
Letting $\rho=\chi_k^{2}$ in Lemma~\ref{lem:R},
\begin{align*}
W_{t}^{\dl}
&\le -2t(\chi_k^{2},\del_{t}\Phi_{t}^{\dl})_\mu+2(\chi_k^{2},\Psi_{t}^{\dl})_\mu-2t \cE^{0}(((\phi')_{t}^{\dl})^{2}\chi_{k}^{2},u_{t}^{\dl}-e_{t}^{\dl}) \\
&\qad +2t\left(\frac{C t}{K}+1\right)W_{t}^{\dl} +t^2a_k+2t R(t,k).
\end{align*}
From (C.3), $(\chi_k^{2},\Psi_{t}^{\dl})_\mu\le KL$.
Moreover,
\begin{align*}
&-2t \cE^{0}(((\phi')_{t}^{\dl})^{2}\chi_{k}^{2},u_{t}^{\dl}-e_{t}^{\dl}) \\
&=-2t \int_E (\phi')_{t}^{\dl} \Gm(\chi_k,\chi_k\phi_{t}^{\dl})\,d\mu
+2t \int_E (\phi')_{t}^{\dl}\phi_{t}^{\dl} \Gm(\chi_k)\,d\mu  
-2t\int_E (\phi'')_{t}^{\dl}(\phi')_{t}^{\dl}\chi_k^{2} \Gm(u_{t}^{\dl}-e_{t}^{\dl})\,d\mu \\
&\le \int_E \left(8t^2\Gm(\chi_k)+\frac18\Gm(\chi_k\phi_{t}^{\dl})\right)\,d\mu 
+2KLt\int_E \Gm(\chi_k)\,d\mu  \\
&\qad +\frac{2Ct}{K}\int_E ((\phi')_{t}^{\dl})^{2}\chi_k^{2} \Gm(u_{t}^{\dl}-e_{t}^{\dl})\,d\mu \qquad\text{(from (C.1), (C.2), and (C.3) )} \\
&=8t^2 a_k+\frac18 U_{t}^{\dl}+2KLt a_{k}+\frac{2Ct}{K}W_{t}^{\dl}.
\end{align*}
Thus, we have
\begin{align*}
W_t^\dl
&\le -2t (\chi_k^{2},\del_{t}\Phi_{t}^{\dl})_\mu
+(9t^2+2KL t)a_k+U_t^\dl/8
+(2Ct/K+2Ct^2/K+2t)W_t^\dl\\
&\qad +2KL+2tR(t,k).
\end{align*}
Choose $T_{0}>0$ such that $2CT_0/K+2CT_0^{2}/K+2T_0 \le 1/2$. 
Then, for $t\in(0,T_0]$,
\begin{equation}\label{eq:Wt}
W_t^\dl 
\le -4t(\chi_k^{2},\del_{t}\Phi_{t}^{\dl})_\mu
+U_{t}^{\dl}/4 +2C_2,
\end{equation}
where $C_2:=(9T_0^2+2KL T_0)a_k+2KL+2T_0R(T_0,k)$.
By putting this inequality
into \eqref{eq:eq20},
\[
U_{t}^{\dl } 
\le -8t(\chi_k^{2},\del_{t}\Phi_{t}^{\dl})_\mu+U_{t}^{\dl}/2 +4C_2+2K^{2}L^{2} a_{k},
\]
so that
\begin{equation}\label{eq:Ut}
U_{t}^{\dl} \le -16t \partial_{t}(\chi_k^{2}, \Phi_{t}^{\dl })_\mu+2C_3,
\end{equation}
where $C_3=4C_2+2K^{2}L^{2} a_{k}$.
Therefore,
\begin{align*}
 \int_{\eps}^{t}\cE^{0}(\phi_{s}^{\dl}\chi_k)\,ds&=\frac{1}{2}\int_{\eps}^{t}U_{s}^{\dl}\,ds \\
 &\le -8 \int_{\eps}^{t} s\partial_{s}(\chi_k^{2}, \Phi_{s}^{\dl } )_\mu\,ds+C_3(t-\eps) \\
 &=-8s \left. \left(\chi_k^{2}, \Phi_{s}^{\dl } \right)_\mu\right|_{s=\eps}^{s=t}+8\int_{\eps}^{t}(\chi_k^{2}, \Phi_{s}^{\dl } )_\mu\,ds+C_3(t-\eps).
\end{align*}
Letting $\eps \to 0$ and dividing by $t$, 
Lemma~\ref{lem:meanlem} gives that
\[
\cE^{0}(\bar{\phi}_{t}^{\dl}\chi_k)
\le \frac{1}{t} \int_{0}^{t} \cE^{0}(\phi_{s}^{\dl}\chi_k)\,ds \le 16KL+C_3.
\]
Therefore, $\left\{\cE^{0}(\bar{\phi}_{t}^{\dl}\chi_k) \right\}_{0<t \le T_{0} , \,0<\dl \le 1}$ is bounded. Moreover, since
\begin{align*}
&2\cE^{0}(\Phi_{t}^{\dl}\chi_k )\\
&=\int_E \Gm(\Phi_{t}^{\dl} \chi_k,\Phi_{t}^{\dl} \chi_k )\,d\mu \\
&=\int_E \Gm({u_{t}^{\dl}-e_{t}^{\dl}})   ( ( \phi' )_{t}^{\dl} )^{4} \chi_k^{2} \,d\mu
+2\int_E \Gm({u_{t}^{\dl}-e_{t}^{\dl},\chi_{k}}) \Phi_{t}^{\dl} (( \phi')_{t}^{\dl} )^{2} \chi_{k} \,d\mu
+\int_E \Gm({\chi_{k}}) ( \Phi_{t}^{\dl})^{2} \,d\mu \\
&\le2\int_E \Gm({u_{t}^{\dl}-e_{t}^{\dl}}) (( \phi' )_{t}^{\dl} )^{4} \chi_k^{2}\,d\mu+2\int_E \Gm({\chi_k})( \Phi_{t}^{\dl} )^{2} \,d\mu  \\
&\le 2W_{t}^{\dl }+2 K^{2}L^{2}a_{k} \qquad\text{(from (C.3))} \\
&\le -16 t\del_t(\chi_k^2,\Phi_t^\dl)_\mu+C_3+4C_2+2 K^{2}L^{2}a_{k}, \quad\text{(from \eqref{eq:Wt} and \eqref{eq:Ut})}
\end{align*}
we can prove the boundedness of $\left\{ \cE^{0}(\bar{\Phi}_{t}^{\dl}\chi_k) \right\}_{0<t \le T_{0},\,0<\dl \le1 }$ in the same way.
\end{proof}
\subsection{Sharper estimates}
We write $u_t=-t\log T_t\bone_B$, $\ph_t=\ph(u_t)$, $\Phi_t=\Phi(u_t)$, and $\Psi_t=\Psi(u_t)$ for $t>0$.
Since $\bar{\phi}_{t}^{\dl}\chi_k$ converges to 
$\bar{\phi}_{t}\chi_k$ $\mu$-a.e.\ as $\dl \to 0$ and $ \{ \bar{\phi}_{t}^{\dl}\chi_k \}_{0\le t\le T_{0},\, 0<\dl \le 1} $ is bounded in $\D$, we conclude that 
$\bar{\phi}_{t}\chi_k \in \D$ and $\left\{ \bar{\phi}_{t}\chi_k\right\}_{0<t\le T_{0}}$ is bounded in $\D$
 for each $k$ by {\cite[Lemma~2.12]{MR}}.
Using the diagonal argument, for any decreasing sequence $\{t_{n}\} \downarrow 0$, we can find a subsequence $\{t_{n'}\}$ such that, for every $k$, $\bar{\phi}_{t_{n'}}\chi_k$ converges weakly to some $\psi_{k}$ in $\D$. 
Since $\chi_k=1$ on $\hat E_{l}$ when $k\geq l$, it follows that $\psi_{k}=\psi_{l}$ $\mu$-a.e.\ on $\hat E_{l}$ for $k \ge l$. 
Therefore, there exists $\bar{\phi}_{0} \in \D_{\loc,b}(\{\hat E_{k}\})$ such that $\psi_{k}=\bar{\phi}_{0}$ $\mu$-a.e.\ on $\hat E_{k}$ for every $k$.

 We may also assume, by taking a further subsequence if necessary, that there exist $\Phi_{0}, \bar{\Phi}_{0}$, and $\bar{\Psi}_{0}$ in $L^{\infty}(\mu)$ such that $\Phi_{t_{n'}} \to \Phi_{0}, \bar{\Phi}_{t_{n'}} \to \bar{\Phi}_{0}, \bar{\Psi}_{t_{n'}} \to \bar{\Psi}_{0}$ both in 
the weak-$L^{2}(\tilde{\mu})$ sense and in the weak${}^{\ast}$-$L^{\infty}(\mu)$ sense. Here, $\tilde{\mu}$ is an arbitrary fixed \emph{finite} measure on $E$ such that $\tilde{\mu}$ and $\mu$ are mutually absolutely continuous, and $L^{\infty}(\mu)$
 is regarded as the dual space of $L^{1}(\mu)$.
We remark that these functions depend on $K$. Because of this, it is more precise to write $\ph_t^K$ and $\bar\Phi_0^K$ instead of $\ph_t$ and $\bar\Phi_0$.

Define
\[
Z_{t}=\{x \in E \mid T_{t}\bone_{B}(x)-\bone_{B}(x)>1 \}
\]
for $t>0$.
From the Chebyshev inequality,
\[
\mu(Z_{t}) \le \|T_{t}\bone_{B}-\bone_{B}\|_{2}^{2} \to 0 \qquad\text{as }t\to0;
\]
thus
\begin{equation}\label{eq:ZZ}
\frac1t\int_0^t \bone_{Z_s}\,ds \to 0 \quad\text{in $L^1(\mu)$ as }t\to0.
\end{equation}
For $x\in Z_t$, 
\[
\phi_t(x)=\phi(-t\log T_t\bone_B(x))<\ph(-t\log 1)=0.
\]
Since $T_{t}\bone_{B}(x)\le 2$ for $x \in E \setminus Z_{t}$ 
and $\phi(y) \ge y$ for $y \le K$, it holds that
\begin{equation}\label{eq:pp}
\phi_{t}(x)=\phi \left(-t \log T_{t}\bone_{B}(x) \right) \geq \phi(-t \log2) \geq -t\log2
\quad\text{on }E\setminus Z_t.
\end{equation}
Similar inequalities hold for $\Phi_t$ and $\Psi_t$.
\begin{lemma} \label{lem:lemorder}
$\bar{\phi}_{0} \ge \bar{\Phi}_{0} \ge \bar{\Psi}_{0}\ge0$ $\mu$-a.e.\ and ${\Phi}_{0} \ge 0$ $\mu$-a.e.
\end{lemma}
\begin{proof}
By using (C.3) and the inequality \eqref{eq:pp} with $\ph_t$ replaced by $\Psi_{t}$,
for $Y \in \cB_{0}$,
\begin{align*}
\int_{Y} \bar{\Psi}_{t}\,d\mu
&=\frac{1}{t}\int_{0}^{t} \left( \int_{Y \cap Z_{s}}\Psi_{s} \,d\mu+\int_{Y \setminus Z_{s}}\Psi_{s}\,d\mu \right)\,ds 
\ge -\frac{LK}{t} \int_{0}^{t} \mu(Z_{s})\,ds-\frac{\mu(Y)\log2}{2}t
\end{align*}
for every $t >0$. Then, by letting $t \to 0$ along the sequence $\left\{ t_{n'}\right\}$ in the above inequality, 
we obtain $\int_{Y} \bar{\Psi}_{0}\,d\mu \ge 0$
by applying \eqref{eq:ZZ}. 
Therefore, $\bar{\Psi}_{0} \ge 0$ $\mu$-a.e.
Next, for $Y \in \cB_{0}$,
\begin{align*}
\int_{Y}(\bar{\phi}_{t}-\bar{\Phi}_{t})\,d\mu &=\frac{1}{t}\int_{0}^{t} \int_{Y \cap Z_{s}}(\phi_{s}-\Phi_{s})\,d\mu \,ds+\frac{1}{t}\int_{0}^{t} \int_{Y \setminus Z_{s}}(\phi_{s}-\Phi_{s})\,d\mu \,ds  \\
&\ge -\frac{2LK}{t} \int_{0}^{t} \mu(Z_{s})\,ds +0
\quad \text{(from (C.3))}
\end{align*}
for $t>0$ small enough. By the same argument, we obtain $\bar{\phi}_{0} \ge \bar{\Phi}_{0}$, and the other inequalities are proved in the same way.
\end{proof}
\begin{lemma} \label{lem:lemphi}
$\bar{\phi}_{0}=0$ $\mu$-a.e.\ on $B$.
\end{lemma}
\begin{proof}
For an arbitrary sequence $\left\{ s_{n} \right\} \downarrow 0$,  
$T_{s_n}\bone_{B}$ converges to $\bone_{B}$ in $L^{2}(\mu)$ as $n \to \infty$. 
Take an arbitrary subsequence $\left\{ s_{n'} \right\}$ from $\left\{s_{n} \right\}$. From this, we can find a subsequence $\left\{ s_{n''}\right\}$ from $\left\{ s_{n'} \right\}$ such that $T_{s_{n''}}\bone_{B} \to \bone_{B}$ $\mu$-a.e.\ as $n'' \to \infty$. Using the dominated convergence theorem, $\lim_{n'' \to \infty} \int_{B}\phi_{s_{n''}}\,d\mu=0$. This means $\lim_{t \to 0}\int_{B}\phi_{t}\,d\mu=0$.
Then, by  letting $t \to 0$ along the sequence $\left\{ t_{n'} \right\}$ in the identity
\[
 \int_{B} \bar{\phi}_{t}d\mu=\frac{1}{t} \int_{0}^{t} \int_{B} \phi_{s}\, d\mu \,ds,
\]
we obtain $\int_{B} \bar{\phi}_{0}\,d\mu=0$. The claim follows directly.
\end{proof}
\begin{lemma} \label{lem:sqlem}
$\displaystyle \sqrt{\bar{\phi}_{0}} \in \D_{0}$. 
\end{lemma}
\begin{proof}
Fix $h \in \D_{\hat{E}_k,b,+}$ arbitrarily. Since $\phi_{t}^{\dl}\chi_k-(\phi_t^\dl - \phi(e_t^\dl))$ is constant on $\hat E_k$, Proposition~\ref{prop:D} implies that 
\[
\int_E \Gm({\phi_{t}^{\dl}\chi_k})h\,d\mu
=\int_E \Gm({\phi_t^\dl - \phi(e_t^\dl)})h\,d\mu
=\int_E \Gm({u_t^\dl-e_t^\dl})((\phi')_t^\dl)^2 h\,d\mu.
\]
By using Lemma~\ref{lem:R} with $\rho=h$, 
\begin{align*}
&\int_E \Gm({u_{t}^{\dl}-e_{t}^{\dl }})((\phi')_{t}^{\dl})^{2}h\,d\mu \\
&\le-2t(h,\del_{t} \Phi_{t}^{\dl})_\mu+2(h,\Psi_{t}^{\dl})_\mu -2t \cE^{0}(((\phi')_{t}^{\dl})^{2}h,u_{t}^{\dl}-e_{t}^{\dl})\\
&\qad+2t\left(\frac{C t}{K}+1\right)\int_E \Gm ( u_{t}^{\dl}-e_{t}^{\dl})((\phi')_{t}^{\dl})^{2}h\,d\mu+2t^2\cE^0(h)+2t R(t,k).
\end{align*}
Since
\begin{align*}
-2t \cE^{0}(((\phi')_{t}^{\dl})^{2}h,u_{t}^{\dl}-e_{t}^{\dl})
&=-t\int_E \Gm ( h, u_t^\dl-e_t^\dl)((\ph')_t^\dl)^2\,d\mu
-2t\int_E \Gm (u_{t}^{\dl}-e_{t}^{\dl}) (\phi')_{t}^{\dl}(\phi'')_{t}^{\dl}h\,d\mu \\
&\le -2t\cE^0( h, \Phi_{t}^{\dl}\chi_k)+\frac{2Ct}{K}\int_E \Gm(\phi_t^\dl \chi_k)h\,d\mu\qquad\text{(from (C.2))},
\end{align*}
we have
\begin{align*}
&\left(1-\frac{2Ct}{K}-\frac{2Ct^2}{K}-2t \right)\int_E \Gm(\phi_t^\dl \chi_k)h\,d\mu \\
&\le -2t (h,\del_{t} \Phi_{t}^{\dl})_\mu + 2(h,\Psi_{t}^{\dl})_\mu-2t \cE^0(h,\Phi_{t}^{\dl}\chi_k) +2t^2\cE^0(h)+2t R(t,k).
\end{align*}
Then, for $T\in(0,T_0]$,
\begin{align}
&\frac{1}{2T}\int_{0}^{T} \left(1-\frac{2CT}{K}-\frac{2CT^2}{K}-2T\right)\int_E \Gm(\phi_t^\dl \chi_k)h\,d\mu\,dt \nonumber \\
&\le -\frac{1}{T}\int_0^T t (h,\del_t\Phi_t^{\dl})_\mu\,dt+ (h,\bar{\Psi}_{T}^{\dl})_\mu-\frac{1}{T}\int_{0}^{T} t \cE^0 (h,\Phi_{t}^{\dl}\chi_k) \,dt +C_4 T,
\label{eq:eq100}
\end{align}
where $C_4=T_0\cE^0(h)/3+ R(T_0,k)/2$.
Integration by parts gives 
\begin{align*}
\frac{1}{T}\int_0^T t (h,\del_t\Phi_t^{\dl})_\mu\,dt
&=\frac1T\left(T(h,\Phi_T^\dl)_\mu-\int_0^T (h,\Phi_t^\dl)_\mu\,dt\right)
=(h,\Phi_{T}^{\dl})_\mu-(h,\bar{\Phi}_{T}^{\dl})_\mu\\
\shortintertext{and}
\frac{1}{T}\int_{0}^{T} t \cE^0 (h,\Phi_{t}^{\dl}\chi_k) \,dt
 &=\left. \frac{t}{T} \int_{0}^{t} \cE^0(h,\Phi_{s}^{\dl}\chi_k)\,ds \right|_{t=0}^{t=T}-\frac{1}{T} \int_{0}^{T}\int_{0}^{t}\cE^0(h,\Phi_{s}^{\dl}\chi_k)\,ds\,dt \\
&=T \cE^0(h,\bar{\Phi}_{T}^{\dl}\chi_k)-\frac{1}{T}\int_{0}^{T}t\cE^0(h,\bar{\Phi}_{t}^{\dl}\chi_k)\,dt\\
&\to0\quad \text{as $\dl \to0$ and $T \to 0$}
\end{align*}
because $\cE^{0}(\bar{\Phi}_{t}^{\dl}\chi_k)$ is bounded in
 $\dl $ and $t$ by Lemma~\ref{lem:bddlem}. 
  Letting $\dl \to 0$ and $T \to 0$ along the sequence $\left\{ t_{n'} \right\}$ in \eqref{eq:eq100}, we obtain
from Lemmas~\ref{lem:wklem} and \ref{lem:meanlem} that 
\begin{equation}
\frac{1}{2}\int_E \Gm({\bar{\phi}_{0}})h\,d\mu
\le - (h,\Phi_{0})_\mu+(h,\bar{\Phi}_{0})_\mu+ (h,\bar{\Psi}_{0})_\mu \le (2\bar{\phi}_{0},h)_\mu. \label{eq:eqsqite}
\end{equation}
The second inequality follows from Lemma~\ref{lem:lemorder}.
Then, for each $\eps>0$,
\[
\int_E \Gm\Bigl({\sqrt{\bar{\phi}_{0}+\eps }-\sqrt{\eps}}\Bigr)h\,d\mu
=\frac{1}{4} \int_E \Gm({\bar{\phi}_{0}})\frac{h}{\bar{\phi}_{0}+\eps}\,d\mu
\le \frac{1}{2} \left(2\bar{\phi}_{0},\frac{h}{\bar{\phi}_{0}+\eps} \right)\le \|h\|_{1}.
\]
This inequality holds for all $h \in \bigcup_{k=1}^{\infty} \D_{\hat{E}_{k},b,+}$. 
Hence, $\bigl|\nb\bigl(\sqrt{\bar{\phi}_{0}+\eps}-\sqrt{\eps}\bigr)\bigr|_H\le1$ $\mu$-a.e.
Next, fix $k\in\N$ and let $f_\eps$ denote $(\sqrt{\bar{\phi}_{0}+\eps}-\sqrt{\eps})\chi_k$ for $\eps>0$.
Then, from the argument above, $|\nb f_\eps|_H\le1$ $\mu$-a.e.\ on $\hat E_k$ and $\{f_\eps\}_{\eps>0}$ is bounded in $\D$. Any weak limit in $\D$ of a subsequence should be $\sqrt{\bar{\phi}_{0}}\chi_k$, and so $|\nb(\sqrt{\bar{\phi}_{0}}\chi_k)|_H\le 1$ $\mu$-a.e.\ on $\hat E_k$. Therefore, $\sqrt{\bar{\phi}_{0}}\in\D_0$.
\end{proof}
From Lemma~\ref{lem:lemphi}, Lemma~\ref{lem:sqlem}, and Proposition~\ref{prop:distance},
we conclude $\bar{\phi}_{0} \le \sd_{B}^{2}$ $\mu$-a.e. 
The multiplicative constant can be further improved since this inequality is not sharp.
\begin{lemma}\label{lem:psilem} 
Given $K>0$, we can choose $M>0$ such that
\begin{equation}\label{eq:PP}
\Phi^{K}(\phi^{M}_{t}\bone_{E\setminus Z_t}) \ge \Psi^{K}_{t},
\quad t\in(0,K/\log 2].
\end{equation}
\end{lemma}
\begin{proof}
Since $\Phi^{K}$ is non-decreasing, $\Psi^{K} \le \Phi^{K}$ on $[-K,\infty)$ and (C.5), we can find $M>0$ such that  
$ \Phi^K(M) \ge \sup_{s \in \R} \Psi^{K}(s)$. 
From \eqref{eq:pp}, $\phi_t^M\ge -t\log 2\ge -K$ on $E\setminus Z_t$.
Then, on $E\setminus Z_t$,
\[
\Phi^{K}(\phi^{M}_{t}\bone_{E\setminus Z_t})
=\Phi^{K}(\phi^{M}_{t})\ge \Phi^{K}(M\wg u_t)
=\Phi^K(M)\wg \Phi^K(u_t)\ge \Psi^{K}(u_t).
\]
On $Z_t$, \eqref{eq:PP} is trivial since the left-hand side is zero and the right-hand side is nonpositive.
\end{proof}
\begin{lemma} \label{lem:itelem} 
If the inequality 
\begin{equation}\label{eq:itelem}
\bar{\phi}_{0}^{K} \le \b \frac{\sd_{B}^{2}}{2}\quad \mu\text{-a.e.\ on }\{\sd_B<N\}
\end{equation}
is true for some $\b>1$ for every $K>0$ and every limit $\bar{\phi}_{0}^{K}$, then
\[
\bar{\phi}_{0}^{K} \le (2-\b^{-1}) \frac{\sd_{B}^{2}}{2}\quad \mu\text{-a.e.\ on }\{\sd_B<N\}.
\]
\end{lemma}
\begin{proof}
Given $K>0$, we take $M$ as in Lemma~\ref{lem:psilem}.
Let $Y\in\cB_0$ with $Y\subset \{\sd_B<N\}$.
Using the convexity of $\Phi^{K}(-t\log(\cdot))$ for small $t$, from Lemma~\ref{lem:convexlem}, we have
\begin{align*}
(\Phi_{t}^{K},\bone_{Y})_\mu&=\int_{Y} \Phi^{K}(-t \log T_{t}\bone_{B})\,d\mu \nonumber 
 \ge \mu(Y)\Phi^{K}\left(-t \log \left(\frac{1}{\mu(Y)}(T_{t}\bone_{B},\bone_{Y})_\mu \right) \right). 
\end{align*}
By the upper estimate \eqref{eq:upper},
\[
\varliminf_{t \to 0} -t \log (T_{t} \bone_{B},\bone_{Y})_\mu=\varliminf_{t \to 0}-t \log P_{t}(Y,B) \ge \frac{\sd(Y,B)^{2}}{2}.
\]
Therefore, in the limit,
\begin{align}
(\Phi_{0}^{K},\bone_{Y})_\mu 
&\ge \mu(Y)\Phi^{K}\left(\frac{\sd(Y,B)^{2}}{2} \right) \notag \\
&=\mu(Y)\essinf_{x \in Y}\Phi^{K}\left( \frac{\sd_{B}(x)^{2}}{2}\right) \nonumber \\
&\ge \frac{\mu(Y)}{\b} \essinf_{x \in Y}\Phi^{K}(\bar{\phi}_{0}^{M}(x))
\label{eq:eqchange}
\end{align}
from the assumption of~\eqref{eq:itelem} and (C.6).
Recall now the function $\hat\Phi^{K}$ that was defined in \eqref{eq:Khat}.
Since $\hat\Phi^{K}$ is 1-Lipschitz and concave on $\R$, 
for $t\in(0,K/\log 2]$ we have
\begin{align*}
\hat\Phi^{K}(\bar{\phi}_{t}^{M})
&=\hat\Phi^{K} \left( \frac{1}{t} \int_{0}^{t} \phi^{M}(u_{s})\bone_{E\setminus Z_s}\,ds 
+\frac{1}{t} \int_{0}^{t} \phi^{M}(u_{s})\bone_{Z_s}\,ds\right)  \\
&\ge \hat\Phi^{K} \left( \frac{1}{t} \int_{0}^{t} \phi^{M}(u_{s})\bone_{E\setminus Z_s}\,ds \right)
-\frac{1}{t} \left|\int_{0}^{t} \phi^{M}(u_{s})\bone_{Z_s}\,ds\right|  \\
&\ge \frac{1}{t} \int_{0}^{t} \hat\Phi^{K} (\phi^{M}(u_{s})\bone_{E\setminus Z_s} )\,ds 
-\frac{LM}{t} \int_{0}^{t} \bone_{Z_s}\,ds\\
&\ge \bar{\Psi}_{t}^{K}-\frac{LM}{t} \int_{0}^{t} \bone_{Z_s}\,ds.
\qquad\text{(from \eqref{eq:pp} and Lemma~\ref{lem:psilem})}
\end{align*}
Take $\{t_n\}\downarrow0$ such that $\bar\ph_{t_n}^M$ converges to $\bar\ph_0^M$ weakly in $L^2(\tilde\mu)$.
From Lemma~\ref{lem:lemorder}, Lemma~\ref{lem:lsclem}, and \eqref{eq:ZZ}, we obtain that
$\Phi^{K}(\bar{\phi}_{0}^{M})=\hat\Phi^{K}(\bar{\phi}_{0}^{M})\ge \bar{\Psi}_{0}^{K}$ $\mu$-a.e.\ 
Combining this inequality and \eqref{eq:eqchange}, we get
\begin{equation}
(\Phi_{0}^{K},\bone_{Y})_\mu \ge \frac{\mu(Y)}{\b} \essinf_{x \in Y}\bar{\Psi}_{0}^{K}(x).\label{eq:contradiction}
\end{equation}
We will prove
\begin{equation}
\Phi_{0}^{K} \ge \b^{-1}\bar{\Psi}_{0}^{K}\quad \mu \mbox{-a.e.\ on }\{\sd_B<N\}.\label{eq:contradiction2}
\end{equation}
Assume for contradiction that there exists some
$Y' \in \cB_0$
and $\eps>0$ such that $Y'\subset \{\sd_B<N\}$ and
$ \Phi_{0}^{K} <\b^{-1}\bar{\Psi}_{0}^{K}-\eps$ on $Y'$. 
Let
\[
Y=\left\{ x \in Y'  \relmiddle| \bar{\Psi}_{0}^{K}(x) \le 
\essinf_{y \in Y'}\bar{\Psi}_{0}^{K}(y)+\frac{\b\eps}{2}  \right\}.
\]
Then, $\mu(Y)>0$ and $ \Phi_{0}^{K}\le \b^{-1}\essinf_{y \in Y}\bar{\Psi}_{0}^{K}(y)-\eps/2$ on $Y$. This contradicts \eqref{eq:contradiction}. 

Combining \eqref{eq:contradiction2} with \eqref{eq:eqsqite} and using Lemma \ref{lem:lemorder}, we obtain
\[
\frac12\int_E \Gm(\bar{\phi}_{0}^{K})h\,d\mu \le -\b^{-1}(h,\bar{\Psi}_{0}^{K})_\mu+  (h,\bar{\Phi}_{0}^{K})_\mu+ (h,\bar{\Psi}_{0}^{K})_\mu  
\le (2-\b^{-1})(\bar{\phi}_{0}^{K},h)_\mu
\]
for every $ h$ expressed by $h=h_1h_2$, where $h_1 \in \bigcup_{k=1}^{\infty}\D_{\hat E_{k},b,+}$ and $h_2=(0\vee k(N-\sd_B))\wg 1$ for some $k\in\N$. 
The claim follows by the same argument after \eqref{eq:eqsqite}.
\end{proof}
By repeated application of Lemma~\ref{lem:itelem}, 
$\bar{\phi}_0 \le \sd_{B}^{2}/2$ $\mu$-a.e.\ on $\{\sd_B<N\}$,
and so $\bar{\Phi}_{0} \le \sd_{B}^{2}/2$ $\mu$-a.e.\ on $\{\sd_B<N\}$ from Lemma~\ref{lem:lemorder}.
To obtain the equality, we modify Lemmas~2.9, 2.13, and 2.14 of \cite{HR} as follows.
\begin{lemma}\label{lem:HR2.9}
For any limit $\Phi_0$ (that is a weak-$L^2(\tilde\mu)$ and weak$^*$-$L^\infty(\mu)$-limit of a subsequence of $\{\Phi_t\}_{t>0}$),
\[
\Phi(\sd_B^2/2)\le \Phi_0\quad\mu\text{-a.e.\ on }\{\sd_B<N\}.
\]
\end{lemma}
\begin{proof}
Let $Y\in\cB_0$ such that $Y\subset\{\sd_B<N\}$.
From the upper estimate~\eqref{eq:upper} and Lemma~\ref{lem:convexlem},
\begin{align*}
\Phi\left(\frac{\sd(Y,B)^2}2\right)
&\le\Phi\left(\varliminf_{t\to0}-t\log P_t(Y,B)\right)\\
&=\Phi\left(\varliminf_{t\to0}-t\log \left(\frac1{\mu(Y)}(T_t\bone_B,\bone_Y)_\mu\right)\right)\\
&=\varliminf_{t\to0}\Phi\left(-t\log \left(\frac1{\mu(Y)}(T_t\bone_B,\bone_Y)_\mu\right)\right)\\
&\le \varliminf_{t\to0}\frac1{\mu(Y)}\int_Y \Phi(-t\log T_t\bone_B)\,d\mu\\
&=\varliminf_{t\to0}\frac1{\mu(Y)}\int_Y \Phi_t\,d\mu\\
&\le \frac1{\mu(Y)}\int_Y \Phi_0\,d\mu.
\end{align*}
Let $Y_\eps=\{\Phi_0\le\Phi(\sd_B^2/2)-\eps\}\cap\{\sd_B<N\}$ for $\eps>0$ and suppose, for the sake of contradiction, that $\mu(Y_\eps)>0$.
Then, $Y'_\eps:=\{x\in Y_\eps\cap E_k\mid \Phi(\sd_B^2(x)/2)\le \Phi(\sd(Y_\eps,B)^2/2)+\eps/2\}$ also has $\mu$-positive measure for sufficiently large $k$.
But
\begin{align*}
\frac1{\mu(Y'_\eps)}\int_{Y'_\eps}\Phi\left(\frac{\sd_B^2}2\right)d\mu
&\le \Phi\left(\frac{\sd(Y'_\eps,B)^2}2\right)+\frac\eps2\\
&\le \frac1{\mu(Y'_\eps)}\int_{Y'_\eps}\Phi_0\,d\mu+\frac\eps2\\
&\le \frac1{\mu(Y'_\eps)}\int_{Y'_\eps}\Phi\left(\frac{\sd_B^2}2\right)d\mu-\frac\eps2,
\end{align*}
which is a contradiction. Therefore, $\mu(Y_\eps)=0$.
Since $\eps>0$ is arbitrary, we obtain the desired assertion.
\end{proof}
\begin{lemma}
For any limit $\bar\Phi_0$ (that is, in particular, a weak-$L^2(\tilde\mu)$ limit of a subsequence of $\{\bar\Phi_t\}_{t>0}$), $\Phi(\sd_{B}^{2}/2)\le \bar{\Phi}_{0}$ $\mu$-a.e.\ on $\{\sd_B<N\}$.
\end{lemma}
\begin{proof}
For any $\rho\in L^2(\tilde\mu)$ with $\rho\ge0$ $\mu$-a.e.\ and $\rho=0$ $\mu$-a.e.\ on $\{\sd_B<N\}$, Lemma~\ref{lem:HR2.9} implies
\[
(\Phi(\sd_B^2/2),\rho)_{\tilde\mu}\le \varliminf_{t\to0}(\Phi_t,\rho)_{\tilde\mu}.
\]
Therefore,
\[
(\Phi(\sd_B^2/2),\rho)_{\tilde\mu}
\le \varliminf_{t\to0}\frac1t\int_0^t (\Phi_s,\rho)_{\tilde\mu}\,ds
=\varliminf_{t\to0} (\bar\Phi_t,\rho)_{\tilde\mu}
\le(\bar\Phi_0,\rho)_{\tilde\mu}.
\]
This implies the desired claim.
\end{proof}
From Lemma~\ref{lem:HR2.9} and (C.4), we have $\bar\Phi_0=\sd_{B}^{2}/2$ on $\{\sd_{B}^{2}/2\le K\}\cap\{\sd_B<N\}=:Y_{K,N}$, independently of the choice of subsequence.
In particular, $\bar\Phi_t\bone_{Y_{K,N}}$ converges to $(\sd_{B}^{2}/2)\bone_{Y_{K,N}}$ weakly in $L^2(\tilde\mu)$ as $t\to 0$.
Furthermore, we have the following
\begin{lemma}
$\bar\Phi_0=\Phi(\sd_{B}^{2}/2)$ on $\{\sd_B<N\}$.
\end{lemma}
\begin{proof}
We specify the dependency on $K$ and write $\Phi_t^K$ instead of $\Phi_t$.
From what we have proven, $\bar\Phi^M_t\bone_{Y_{M,N}}$ converges to $(\sd_{B}^{2}/2)\bone_{Y_{M,N}}$ weakly in $L^2(\tilde\mu)$ and $\bar\Phi^M_0=\sd_{B}^{2}/2$ on $Y_{M,N}$ for all $M>0$.
From (C.7) and the 1-Lipschitz continuity of $\Phi^K$,
\begin{align}
\bar\Phi^K_T
&=\frac1T\int_0^T \Phi^K(u_t)\,dt\notag\\
&\le \frac1T\int_0^T \Phi^K(\Phi^M(u_t))\,dt +(\Phi^K(\infty)-\Phi^K(M))\notag\\
&\le \frac1T\int_0^T \Phi^K(\Phi^M(u_t)\bone_{E\setminus Z_t})\,dt +\frac1T\int_0^T|\Phi^M(u_t)\bone_{Z_t}|\,dt+(\Phi^K(\infty)-\Phi^K(M)) \notag \\
&\le \frac1T\int_0^T \Phi^K(\Phi^M(u_t)\bone_{E\setminus Z_t})\,dt +\frac{LM
}{T}\int_0^T\bone_{Z_t}\,dt+(\Phi^K(\infty)-\Phi^K(M)).
\label{eq:HR2.14a}
\end{align}
Since $\hat\Phi^K$ is concave and 1-Lipschitz, the first term is estimated as follows:
\begin{align}
\frac1T\int_0^T \Phi^K(\Phi^M(u_t)\bone_{E\setminus Z_t})\,dt
&=\frac1T\int_0^T \hat\Phi^K(\Phi^M(u_t)\bone_{E\setminus Z_t})\,dt\notag\\
&\le \hat\Phi^K\left(\frac1T\int_0^T \Phi^M(u_t)\bone_{E\setminus Z_t}\,dt\right)\notag\\
&\le \hat\Phi^K\left(\frac1T\int_0^T \Phi^M(u_t)\,dt\right)+\left|\frac1T\int_0^T \Phi^M(u_t)\bone_{Z_t}\,dt\right|\notag\\
&\le \hat\Phi^K(\bar\Phi^M_T)+\frac{LM}T\int_0^T \bone_{Z_t}\,dt.
\label{eq:HR2.14b}
\end{align}
Combining \eqref{eq:HR2.14a} and \eqref{eq:HR2.14b} and letting $T\to0$ along a suitable subsequence, we have
\[
\bar\Phi^K_0 \le \hat\Phi^K(\sd_B^2/2)+(\Phi^K(\infty)-\Phi^K(M))
\quad\text{on }Y_{M,N}.
\]
Here, we also used Lemma~\ref{lem:lsclem} and \eqref{eq:ZZ} for the right-hand side.
Letting $M\to\infty$, we obtain $\bar\Phi^K_0 \le \Phi^K(\sd_B^2/2)$ on $\{\sd_B<N\}$.
The inequality in the other direction has been proven already.
\end{proof}
Therefore, $\bar\Phi_{t}\bone_{\{\sd_B<N\}} $ converges both in the weak-$L^{2}(\tilde{\mu})$ sense and in the weak${}^{\ast}$-$L^{\infty}(\mu)$ sense as $t \to 0$, and the limit $\bar{\Phi}_{0}\bone_{\{\sd_B<N\}}$ is equal to $\Phi(\sd_{B}^{2}/2)\bone_{\{\sd_B<N\}}$.

\subsection{Application of the Tauberian theorem}
 In the following, $\{ T_{t}^{0}\}_{t>0}$ denotes the Markovian semigroup corresponding to the strong local Dirichlet form $(\cE^{0},\D)$, and $\cL^{0}$ denotes the generator of $\{ T_t^0\}_{t>0}$ on $L^2(\mu)$. 
Note that $\cE^0(T_t^0 f) \le (2et)^{-1}\|f\|_2^2$ holds for all $f \in L^2(\mu)$ and $t>0$, from the spectral decomposition theorem.
For $k\in\N$, we define 
\[
\hat\chi_k=\chi_k\cdot((0\vee k(N-\sd_B))\wg1)\in\D_{E_k\cap\{\sd_B<N\},b}.
\]
We use $(\cdot,\cdot)_{\hat\chi_k^2\cdot\mu}$ to denote the $L^{2}$-inner product with respect to the measure $\hat\chi_k^{2}\cdot\mu$ on $E$. 

 \begin{lemma} \label{lem:wiener2}
For $\tau \in (0,1)$ and $Y \in \cB_{0}$,
\begin{align*}
\lim_{t \to 0} (T^{0}_{\tau-t}\bone_{Y},\Phi_{t} )_{\hat\chi_k^2\cdot\mu}=(T^{0}_{\tau} \bone_{Y},\bar{\Phi}_{0})_{\hat\chi_k^2\cdot\mu}.
\end{align*}
\end{lemma}
\begin{proof}
Let $f(t)=\bigl(T^{0}_{\tau-t}\bone_{Y},\Phi_{t}-\Phi(\infty)\bigr)_{\hat\chi_k^2\cdot\mu}$
for $t\in(0,\tau)$.
We will confirm the following two conditions:
\begin{enumerate}
\item $T^{-1}\int_{0}^{T}f(s)\,ds$ converges to $\bigl(T_{\tau}^{0}\bone_{Y},\bar{\Phi}_{0}-\Phi(\infty)\bigr)_{\hat\chi_k^2\cdot\mu}$ as $T \to 0$,
\item there exist $M>0$ and $t_0\in(0,\tau)$ such that $f(t)-f(s) \le M(t-s)/s$ for any $0<s<t\le t_{0}$.
\end{enumerate}
Under these conditions, we can apply  
Wiener's Tauberian theorem (see, e.g., \cite[Lemma 3.11]{R})
to obtain 
$\lim_{t\to 0}f(t)=(T_{\tau}^0\bone_{Y},\bar{\Phi}_{0}-\Phi(\infty))_{\hat\chi_k^2\cdot\mu}$. Combining this equality and the identity 
\[
\lim_{t \to 0}(T_{\tau-t}^{0}\bone_{Y},\Phi(\infty))_{\hat\chi_k^2\cdot\mu}=(T_{\tau}^{0}\bone_{Y},\Phi(\infty))_{\hat\chi_k^2\cdot\mu},\] 
we obtain the desired claim.

Condition (i) is proved as follows:
\begin{align*}
 &\left|\frac{1}{T} \int_{0}^{T} f(t)\,dt-\bigl(T_{\tau}^{0}\bone_{Y},\bar{\Phi}_{0}-\Phi(\infty) \bigr)_{\hat\chi_k^2\cdot\mu}  \right| \\ 
  &\le \frac{1}{T} \int_{0}^{T}\left|\bigl( T_{\tau-t}^{0}\bone_{Y}-T_{\tau}^{0}\bone_{Y}, \Phi_{t} -\Phi(\infty) \bigr)_{\hat\chi_k^2\cdot\mu} \right|\,dt + \left| \bigl(T_{\tau}^{0}\bone_{Y}, \bar{\Phi}_{T}-\bar{\Phi}_{0}  \bigr)_{\hat\chi_k^2\cdot\mu} \right| \\
   &\le \frac{2KL}{T}\int_{0}^{T} \|\bone_{Y}-T_{t}^{0}\bone_{Y} \|_{1}\, dt + \left| \bigl(T_{\tau}^{0}\bone_{Y}, \bar{\Phi}_{T}-\bar{\Phi}_{0}  \bigr)_{\hat\chi_k^2\cdot\mu} \right| \\
    & \to 0 \qquad\text{as }T\to0.
\end{align*}
Here, we used the fact that $\{T_{t}^{0}\}_{t>0}$ extends to a strongly continuous semigroup on $L^1(\mu)$. For condition (ii), we fix $\dl \in (0,1)$.
The function $f^\dl(r):=\bigl( T^{0}_{\tau-r}\bone_{Y},\Phi_{r}^{\dl}-\Phi(e_{r}^{\dl}) \bigr)_{\hat\chi_k^2\cdot\mu}$
is continuously differentiable on $(0,\tau)$ and  
 \begin{align*}
 \frac{d}{dr}f^\dl(r)
&=(T^{0}_{\tau-r}\bone_{Y},\del_{r}\Phi_{r}^{\dl}-\del_{r}\Phi(e_{r}^{\dl}))_{\hat\chi_k^2\cdot\mu}
+ \bigl(\del_{r}T^{0}_{\tau-r}\bone_{Y},\Phi_{r}^{\dl}-\Phi(e_{r}^{\dl})\bigr)_{\hat\chi_k^2\cdot\mu}\\
&=(T^{0}_{\tau-r}\bone_{Y},\del_{r}\Phi_{r}^{\dl})_{\hat\chi_k^2\cdot\mu}- \frac{1}{r} (T^{0}_{\tau-r}\bone_{Y},\Psi(e_{r}^{\dl}))_{\hat\chi_k^2\cdot\mu} - (\cL^0T^{0}_{\tau-r}\bone_{Y},\Phi_{r}^{\dl}-\Phi(e_{r}^{\dl}))_{\hat\chi_k^2\cdot\mu}\\
&\le(T^{0}_{\tau-r}\bone_{Y},\del_{r}\Phi_{r}^{\dl})_{\hat\chi_k^2\cdot\mu}
+ \cE^0(T^{0}_{\tau-r}\bone_{Y},\hat\chi_k^2(\Phi_{r}^{\dl}-\Phi(e_{r}^{\dl})))
=:J_1+J_2.
\end{align*}
Let $r\in(0,\tau/2]$ and  $Q_r=\int_E \Gm ( u_{r}^{\dl}-e_{r}^{\dl})((\phi')_{r}^{\dl})^{2}\hat\chi_k^{2}T_{\tau-r}^{0}\bone_{Y} \,d \mu$.
By using Lemma~\ref{lem:R} with $f=\bone_{B}$ and $ \rho=\hat\chi_k^{2}T_{\tau-r}^{0}\bone_{Y}$,
\begin{align*}
J_{1}
&\le\frac{1}{r}(\hat\chi_k^{2}T_{\tau-r}^{0}\bone_{Y},\Psi_{r}^{\dl})_\mu
-  \cE^{0} \bigl( ((\phi')_{r}^{\dl})^{2} \hat\chi_k^{2}T_{\tau-r}^{0}\bone_{Y}, u_{r}^{\dl}-e_{r}^{\dl} \bigr)\\
&\qad-\frac{1}{2r} \int_E \Gm ( u_{r}^{\dl}-e_{r}^{\dl})((\phi')_{r}^{\dl})^{2}\hat\chi_k^{2}T_{\tau-r}^{0}\bone_{Y} \,d \mu \\
&\qad 
+\left(\frac{Cr}{K}+1 \right)\int_E \Gm ( u_{r}^{\dl}-e_{r}^{\dl})((\phi')_{r}^{\dl})^{2}\hat\chi_k^{2}T_{\tau-r}^{0}\bone_{Y} \,d \mu +r\cE^0(\hat\chi_k^{2}T_{\tau-r}^{0}\bone_{Y}) + R(r,k)\\
&\le \frac{C_5}{r}-\cE^{0} \bigl( ((\phi')_{r}^{\dl})^{2} \hat\chi_k^{2}T_{\tau-r}^{0}\bone_{Y}, u_{r}^{\dl}-e_{r}^{\dl} \bigr)+
\left(\frac{Cr}{K}+1-\frac{1}{2r}\right) Q_r+C_6,
\end{align*}
where $C_5=KL \mu(Y)$ and $C_6=(\tau/2)\sup_{s\in[\tau/2,\tau)}\cE^0(\hat\chi_k^2T_s^0 \bone_Y)+R(\tau/2,k)$.
We also have
\begin{align*}
&\cE^{0} \bigl( ((\phi')_{r}^{\dl})^{2} \hat\chi_k^{2}T_{\tau-r}^{0}\bone_{Y}, u_{r}^{\dl}-e_{r}^{\dl} \bigr)\\
&=\frac{1}{2} \int_E ((\phi')_{r}^{\dl})^{2}\hat\chi_k^{2} \Gm \bigl(T^{0}_{\tau-r}\bone_{Y}, u_{r}^{\dl}-e_{r}^{\dl} \bigr)\,d\mu  +  \int_E \hat\chi_k ((\phi')_{r}^{\dl})^{2} T^{0}_{\tau-r}\bone_{Y}   \Gm \bigl( \hat\chi_k, u_{r}^{\dl}-e_{r}^{\dl}\bigr) \,d\mu  \\*
&\qad + \int_E \hat\chi_k^{2} (\phi')_{r}^{\dl} (\phi'')_{r}^{\dl} T^{0}_{\tau-r}\bone_{Y}  \Gm (u_{r}^{\dl}-e_{r}^{\dl} ) \,d\mu  \\ 
&=: J_3+J_4+J_5.
\end{align*}
Accordingly,
\begin{align*}
\frac{d}{dr}f^\dl(r)
&\le J_2-J_3-J_4-J_5+\frac{C_5}r+\left(\frac{Cr}{K}+1-\frac{1}{2r} \right)Q_r+C_6.
\end{align*}
Then,
\begin{align*}
J_2-J_3
&=\frac{1}{2}  \int_E \Gm({T^{0}_{\tau-r}\bone_{Y}, \hat\chi_k^{2}})\bigl(\Phi_{r}^{\dl}-\Phi(e_{r}^{\dl}) \bigr)d\mu \\
&\le 2KL\cE^0(T^{0}_{\tau-r}\bone_{Y})^{1/2}\cE^0(\hat\chi_k^{2})^{1/2}\\
&\le4KL(e\tau)^{-1/2}\mu(Y)^{1/2}\cE^0(\hat\chi_k)^{1/2} 
=: C_7,
\end{align*}
\[
|J_4|
\le \frac12 \int_E \left\{\Gm(\hat\chi_k)+\hat\chi_k^2((\phi')_{r}^{\dl})^{2}\Gm(u_{r}^{\dl}-e_{r}^{\dl})\right\}T^{0}_{\tau-r}\bone_{Y}\, d\mu
\le  \cE^0(\hat\chi_k)+ \frac{Q_r}{2},
\]
and
\[
|J_5|\le \frac{C}{K}Q_r.
\]
Therefore, we have
\[
\frac{d}{dr}f^\dl(r)\le \left(\frac{C+Cr}{K}+\frac32-\frac1{2r}\right)Q_r+\frac{C_5}{r}+C_6+C_7+\cE^0(\hat\chi_k).
\]
Let $C_8$ denote $C_6+C_7+\cE^0(\hat\chi_k)$.
If $r\le t_0:=\min\{(2(C\tau+C)/K+3)^{-1},\tau/2\}$, then we have $\frac{d}{dr}f^\dl(r)\le C_5/r+C_8$.
From this,
\[
f^\dl(t)-f^\dl(s)
\le  \int_s^t \left(\frac{C_5}{r}+C_8\right)dr
\le \frac{C_5+C_8 t_0 }{s}(t-s),
\quad 0<s<t\le t_0.
\]
By letting $\dl\to0$, we obtain the desired assertion.
\end{proof}
Let $\tau>0$, $k\in\N$, and $Y \in \cB_{0}$. 
It holds that
\begin{equation}
\int_{Y}\hat\chi_k^{2}\Phi_{t}\,d \mu=(\Phi_{t}, T_{\tau-t}^{0} \bone_{Y})_{\hat\chi_k^2\cdot\mu}
+(\hat\chi_k^{2}\Phi_{t}, \bone_{Y}-T_{\tau-t}^{0}\bone_{Y})_\mu \label{eq:ast}
\end{equation}
for every $t \in (0,\tau]$. From Lemma~\ref{lem:wiener2}, 
\[
\lim_{\tau \to0} \lim_{t \to 0} (\Phi_{t}, T_{\tau-t}^0 \bone_{Y})_{\hat\chi_k^2\cdot\mu}
=(\bar{\Phi}_{0},\bone_{Y})_{\hat\chi_k^2\cdot\mu}.
\]
For the second term of \eqref{eq:ast}, we have 
\[
\lim_{\tau \to 0} \lim_{t \to 0}(\hat\chi_k^{2}\Phi_{t},\bone_{Y}-T^{0}_{\tau-t}\bone_{Y})_\mu=0
\]
by the boundedness of $\Phi_{t}$ and the strong continuity of $\{T_{t}^{0} \}$. Therefore, 
\begin{equation}
\lim_{t \to 0}\int_{Y} \hat\chi_k^{2}\Phi_{t} \,d\mu=\int_{Y} \hat\chi_k^{2}\bar{\Phi}_{0}\,d\mu  = \int_{Y} \hat\chi_k^{2}\Phi (\sd_{B}^{2}/2) \,d\mu. \label{eq:ast2}
\end{equation}
Using this equality, we finish the proof of Theorem~\ref{th:lower}. 
For $\eps\in(0,(N-\sd(A,B))/2)$, define
 $A_\eps=\{x \in A \mid \sd_{B}(x) \le \sd(A,B)+\eps\}$. 
Then, for sufficiently large $k$, $\hat\chi_k=\chi_k=1$ $\mu$-a.e.\ on $A_\eps\cap\hat E_k$.
 Using the convexity of $\Phi^{K}(-t\log(\cdot))$ for small $t$
(Lemma~\ref{lem:convexlem}), we have
 \begin{align*}
\varlimsup_{t \to 0} \Phi (-t\log P_{t}(A,B)) 
 &\le \varlimsup_{t \to 0} \Phi (-t\log P_{t}(A_\eps, B) ) \\
&\le \varlimsup_{t \to 0}\frac{1}{\mu( A_\eps)} \int_{ A_\eps} \Phi(-t\log T_{t}\bone_{B} )\,d\mu \\
&= \varlimsup_{t \to 0}  \frac{1}{\mu( A_\eps)}\left( \int_{ A_\eps \cap \hat E_{k}}\hat\chi_k^{2}\Phi_{t} \,d\mu+ \int_{ A_\eps \setminus \hat E_{k} } \Phi_{t}\,d\mu \right) \\
&\le \varlimsup_{t \to 0}  \frac{1}{\mu(A_\eps)} \int_{ A_\eps} \hat\chi_k^{2}\Phi_{t} \,d\mu+\frac{KL}{\mu( A_\eps)} \mu(A_\eps\setminus \hat E_{k})\\
&= \frac{1}{\mu( A_\eps)} \int_{A_\eps } \hat\chi_k^{2} \Phi \left( \frac{\sd_{B}^{2}}{2} \right)\,d\mu +\frac{KL}{\mu(A_\eps)} \mu( A_\eps\setminus \hat E_{k}) \qquad\text{(from \eqref{eq:ast2})} \\
&\le  \frac{\left(\sd(A,B)+\eps \right)^{2}}{2} +\frac{KL}{\mu( A_\eps)} \mu( A_\eps\setminus \hat E_{k}).
\end{align*}
Letting $k \to \infty$, we have
\[
\varlimsup_{t \to 0} \Phi (-t\log P_{t}(A,B) ) \le \frac{\left(\sd(A,B)+\eps \right)^{2}}{2}.
\]	
Since $\eps$ is arbitrary and $\lim_{K \to \infty}\Phi^{K}(x)=x$ for each $x \in \R$, \eqref{eq:lower} holds.
This completes the proof of Theorem~\ref{th:lower}.
\section{Proof of auxiliary propositions and examples}
We prove Propositions~\ref{prop:infty} and \ref{prop:B2}.
\begin{proof}[Proof of Proposition~\ref{prop:infty}]
Let $Y=\{\sd_B=\infty\}$. From \cite[Proposition~5.1]{AH}, $Y$ coincides with $\{T^0_t \bone_B=0\}$ up to $\mu$-null sets for all $t>0$.
Then, for $t>0$ and $f\in L^2(\mu)$ with $f\ge0$ $\mu$-a.e.,
\[
  (T^0_t(\bone_Y f),T^0_1\bone_B)_\mu=(\bone_Y f,T^0_{t+1}\bone_B)_\mu=0.
\]
This implies that 
\begin{equation}\label{eq:invariant}
T^0_t(\bone_Y f)=0\quad\mu\text{-a.e.\ on }E\setminus Y
\end{equation}
because $T^0_1\bone_B>0$ $\mu$-a.e.\ on $E\setminus Y$.
Equation \eqref{eq:invariant} now holds for all $f\in L^2(\mu)$; thus, $Y$ is an invariant set with respect to $(\cE^0,\D)$, from \cite[Lemma~1.6.1]{FOT}.
From \cite[Theorem~1.6.1]{FOT}, $\bone_Y f\in\D$ for every $f\in\D$ and 
\[
\cE^0(f,g)=\cE^0(\bone_Y f,\bone_Y g)+\cE^0(\bone_{E\setminus Y} f,\bone_{E\setminus Y} g),\quad f,g\in\D.
\]
Moreover, $\nb(\bone_Y f)=\bone_Y\nb f$ by Proposition~\ref{prop:D}. 
By using these properties, we can confirm that $\cE(\bone_{E\setminus Y} f,\bone_Y g)=0$ for $f,g\in\D$.
From \cite[Lemma~3.1]{Kuw}, $E\setminus Y$ is weakly invariant with respect to $(\cE,\D)$. That is, $T_t (\bone_{E\setminus Y}f)=0$ $\mu$-a.e.\ on $Y$ for all $t>0$ and $f\in L^2(\mu)$. Letting $f=\bone_B$, we obtain $T_t\bone_B=0$ $\mu$-a.e.\ on $Y$ since $B \subset E \setminus Y$.
Thus, for $A\in\cB_0$ satisfying $\sd(A,B)=\infty$, we have $\mu(A\setminus Y)=0$ and so $P_t(A,B)=0$ holds.
\end{proof}
\begin{proof}[Proof of Proposition~\ref{prop:B2}]
Let $f\in \bigcup_{k=1}^\infty \D_{E_k,b}$ with $\|f\|_2=1$.
For $x\ge0$ and $y\ge0$,
\begin{align*}
xy
&\le \int_0^x \log^+ s\,ds+\int_0^y e^t\,dt\\
&\le \int_0^x \log^+ x\,ds+e^y-1\\
&= x\log^+ x+e^y-1.
\end{align*}
Applying this inequality to $x=f^2$ and $y=\dl |b_2-c_2|_H^2$ and integrating it provides 
\[
\dl\int_E |b_2-c_2|_H^2 f^2\,d\mu
\le \int_E f^2\log^+{f^2}\,d\mu
+\int_E \bigl( e^{\dl|b_2-c_2|_H^2}-1 \bigr)d\mu.
\]
Then,
\begin{align*}
\int_E |b_2-c_2|_H f^2\,d\mu
&\le\left(\int_E |b_2-c_2|_H^2 f^2\,d\mu\right)^{1/2}\\
&\le \dl^{-1/2}\left(\int_E f^2\log^+{f^2}\,d\mu\right)^{1/2}
+\dl^{-1/2}\left\{\int_E \bigl( e^{\dl|b_2-c_2|_H^2}-1 \bigr)d\mu\right\}^{1/2}.
\end{align*}
Thus, (B.2$'$) holds with $\gm=\dl^{-1/2}$ and $\lm_\eps=\hat\lm_\eps+\dl^{-1/2}\left\{\int_E \bigl( e^{\dl|b_2-c_2|_H^2}-1 \bigr)d\mu\right\}^{1/2}$ for $\eps>0$.
\end{proof}
We discuss other sufficient conditions of Assumption~\ref{as:1}, (B.1), (B.2)$_{A,B}$, and (B.2$'$).

For $\a>0$, let $\cT_\a$ denote the set of all real measurable functions $\psi$ on $E$ satisfying the following requirement: there exists a measurable nest $\{E_k\}_{k=1}^\infty$ and $\tilde\lm_\a\ge0$ such that $\psi\in L^1_\loc(\mu,\{E_k\})$ and 
\begin{equation}\label{eq:cTa}
\int_E |\psi| f^2\,d\mu \le \a \cE^0(f)+\tilde\lm_\a\|f\|_2^2,\quad f\in\bigcup_{k=1}^\infty \D_{E_k,b}.
\end{equation}
Also, let $\cT_{0+}=\bigcap_{\a>0}\cT_\a$ and $\cT_\infty=\bigcup_{\a>0}\cT_\a$.
Clearly, $L^\infty(\mu)\subset \cT_{0+}$.
\begin{lemma}\label{lem:cT}
Let $a$ be an $H$-valued measurable function on $E$, and $\a>0$. 
If $|a|_H^2\in \cT_\a$ with a measurable nest $\{E_k\}_{k=1}^\infty$ and $\tilde\lm_\a\ge0$, then the following inequalities hold for $f,g\in\bigcup_{k=1}^\infty \D_{E_k,b}$ and $\eps>0$:
\begin{align}
\left|\int_E (a,\nb f)_H g\,d\mu\right|&\le\left(2\cE^0(f)\right)^{1/2}\left(\a\cE^0(g)+\tilde\lm_\a\|g\|_2^2\right)^{1/2},\label{eq:cT1}\\
\left|\int_E (a,\nb f)_H f\,d\mu\right|&\le\sqrt{2\a}\cE^0(f)+\frac{\tilde\lm_\a}{\sqrt{2\a}}\|f\|_2^2,\label{eq:cT2}\\
\int_E |a|_H f^2\,d\mu&\le \eps\cE^0(f)+\left(\frac{\eps \tilde\lm_\a}{\a}+\frac{\a}{4\eps}\right)\|f\|_2^2.\label{eq:cT3}
\end{align}
\end{lemma}
\begin{proof}
Equation~\eqref{eq:cT1} holds by combining the inequality
\[
\left|\int_E (a,\nb f)_H g\,d\mu\right|
\le \left(\int_E |\nb f|_H^2\,d\mu\right)^{1/2}\left(\int_E |a|_H^2 g^2\,d\mu\right)^{1/2}
\]
with \eqref{eq:cTa}.
By letting $g=f$,
\begin{align*}
\left|\int_E (a,\nb f)_H f\,d\mu\right|
&\le\left(2\cE^0(f)\right)^{1/2}\left(\a\cE^0(f)+\tilde\lm_\a\|f\|_2^2\right)^{1/2}\\
&\le\frac{\sqrt{2\a}}4\cdot2\cE^0(f)+\frac{1}{\sqrt{2\a}}\left(\a\cE^0(f)+\tilde\lm_\a\|f\|_2^2\right),
\end{align*}
which proves \eqref{eq:cT2}.
Equation~\eqref{eq:cT3} follows from the following calculation:
\begin{align*}
\int_E |a|_H f^2\,d\mu
&\le \frac{\eps}{\a}\int_E |a|_H^2 f^2\,d\mu+\frac{\a}{4\eps}\int_E f^2\,d\mu\\
&\le \frac{\eps}{\a}\left(\a\cE^0(f)+\tilde\lm_\a\|f\|_2^2\right)+\frac{\a}{4\eps}\|f\|_2^2.\qedhere
\end{align*}
\end{proof}
We can give sufficient conditions for Assumption~\ref{as:1}, (B.1), and (B.2$'$)  in terms of $\cT_\a$.
\begin{prop}\label{prop:suff}
Suppose $|b|_H,|c|_H\in L^2_\loc(\mu,\{E_k\})$ and $V\in L^1_\loc(\mu,\{E_k\})$ for some measurable nest $\{E_k\}_{k=1}^\infty$. 
Let $V_+$ denote $V\vee0$ and $V_-$ denote $(-V)\vee0$.
\begin{enumerate}
\item If $|b+c|_H^2\in \cT_{\a_1}$, $|b-c|_H^2\in\cT_{\a_2}$, $V_+\in\cT_{\a_3}$, $V_-\in \cT_{\a_4}$ with $\a_i>0$ $(i=1,2,3,4)$ and $\sqrt{2\a_1}+\a_4<1$, 
then {\rm(A.2)} holds.
\item If $|b-c|_H^2\in\cT_\infty$, then {\rm (B.1)} holds.
\item If $b$ and $c$ are decomposed into $b=b_1+b_2$ and $c=c_1+c_2$, respectively, such that $|b_1-c_1|_H^2\in\cT_{0+}$ and $\exp(\dl|b_2-c_2|_H^2)-1\in L^1(\mu)$ for some $\dl>0$, then {\rm (B.2$'$)} holds.
\end{enumerate}
\end{prop}
\begin{proof}
(i): Equation~\eqref{eq:A2_1} follows from \eqref{eq:cT2}. Since $|b|_H^2+|c|_H^2=(|b+c|_H^2+|b-c|_H^2)/2$, we see that $|b|_H^2, |c|_H^2\in \cT_{(\a_1+\a_2)/2}$.
Moreover, $|V|\in \cT_{\a_3+\a_4}$. Combining these and \eqref{eq:cT1}, we obtain \eqref{eq:A2_2}.

\noindent (ii): This follows from \eqref{eq:cT2}.

\noindent (iii): From \eqref{eq:cT3}, Eq.\,\eqref{eq:upperB1'} holds
with $\hat\lm_\eps=\inf_{\a>0}\left(\frac{\eps \tilde\lm_\a}{\a}+\frac{\a}{4\eps}\right)$ for $\eps>0$.
Moreover, for any $\a>0$,
\[
\varlimsup_{\eps\to0}\eps\hat\lm_\eps
\le \varlimsup_{\eps\to0}\eps\biggl(\frac{\eps \tilde\lm_\a}{\a}+\frac{\a}{4\eps}\biggr)
=\frac\a4,
\]
which implies that $\varlimsup_{\eps\to0}\eps\hat\lm_\eps=0$. 
Therefore, (B.2$'$) holds, from Proposition~\ref{prop:B2}.
\end{proof}

The following are alternative descriptions of (B.2)$_{A,B}$ and (B.2$'$).
\begin{prop}
{\rm(B.2)$_{A,B}$} is equivalent to the following condition:
\begin{enumerate}[(\~ B.2)$_{A,B}$]
\item[\rm(\~ B.2)$_{A,B}$]There exist $\gm\ge0$ and a nonnegative and non-decreasing function $\Om$ on $[0,\infty)$ such that $\lim_{x\to\infty}\Om(x)/x=0$ and
\begin{align}
&\int_{\{0<\sd_B<\sd(A,B)\}} (b-c, \nb \sd_B)_H f^2\,d\mu\le \Om\left(\cE^0(f)^{1/2}\right)+\gm\left(\int_E f^2\log^+ f^2\,d\mu\right)^{1/2}\notag\\
&\hspace{9em}\text{for any $N>0$ and } f\in\bigcup_{k=1}^\infty \D_{E_k,b}\text{ with }\|f\|_2=1.
\label{eq:tB2}
\end{align}
\end{enumerate}
Moreover, {\rm(B.2$'$)} is equivalent to the following condition:
\begin{enumerate}[\rm(\~ B.2$'$)]
\item[\rm(\~ B.2$'$)]There exist $\gm\ge0$ and a nonnegative and non-decreasing function $\Om$ on $[0,\infty)$ such that $\lim_{x\to\infty}\Om(x)/x=0$ and
\begin{align*}
\int_E |b-c|_H f^2\,d\mu&\le \Om\left(\cE^0(f)^{1/2}\right)+\gm\left(\int_E f^2\log^+ f^2\,d\mu\right)^{1/2}\notag\\
&\qquad\text{for any } f\in\bigcup_{k=1}^\infty \D_{E_k,b}\text{ with }\|f\|_2=1.
\end{align*}
\end{enumerate}
\end{prop}
\begin{proof}
Suppose (B.2)$_{A,B}$ holds.
For $t\ge0$, we define $\Om(t)=\inf_{\eps>0}(\eps t^2+\lm_\eps)$.
(B.2)$_{A,B}$ implies \eqref{eq:tB2}.
Moreover, $\Om$ is non-decreasing by definition. 
For any $\a>0$, we have
\begin{align*}
\frac{\Om(t)}{t}=\inf_{\eps>0}\left(\eps t+\frac{\lm_\eps}{t}\right)
\le \a+\frac{\lm_{\a/t}}{t}
\xrightarrow{t\to\infty}\a.
\end{align*}
Since $\a>0$ is arbitrary, we conclude that $\Om(t)/t$ converges to $0$ as $t\to\infty$. Thus, (\~B.2)$_{A,B}$ holds.

Conversely, assume (\~B.2)$_{A,B}$.
Define $\lm_\eps=\sup_{t>0}(\Om(t^{1/2})-\eps t)$ for each $\eps>0$.
Then, \eqref{eq:B2} holds.
Furthermore, for any $\a>0$, there exists some $T>0$ such that $\Om(t^{1/2})/t^{1/2}\le \a$ for all $t\ge T$. 
Then,
\begin{align*}
\eps\lm_\eps
&\le \sup_{t\in[0,T)}(\eps\Om(t^{1/2})-\eps^2 t)\vee \sup_{t\in[T,\infty)}(\eps\Om(t^{1/2})-\eps^2 t)\\
&\le \eps\Om(T^{1/2})\vee \sup_{t\in[T,\infty)}(\eps\a t^{1/2}-\eps^2 t)\\
&\le \eps\Om(T^{1/2})\vee (\a^2/4).
\end{align*}
Letting $\eps\to0$ and $\a\to0$, we have $\lim_{\eps\to0}\eps\lm_\eps=0$.
Therefore, (B.2)$_{A,B}$ holds.

The equivalence of (B.2$'$) and (\~B.2$'$) is proved in the same way.
\end{proof}

We discuss some typical examples of non-symmetric forms that satisfy Assumption~\ref{as:1}, (B.1), and (B.2$'$). 
\begin{example}\label{ex:1}
We assume the Sobolev inequality: for some $d>2$ and $S \ge 0$,
\begin{equation} \label{eq:sobolev}
\|f\|^{2}_{2d/(d-2)} \le S \, \cE^{0}_1(f), \quad f \in \D.
\end{equation}
A typical example is the following: $E=\R^d$; $\mu=dx$ (Lebesgue measure); $\D=H^1(\R^d)$ (the first-order $L^2$-Sobolev space); and $D=A(x)\nabla$, where $A(\cdot)$ is an $\R^{d\times d}$-valued measurable function on $E$ such that there exist $C>0$ satisfying $C^{-1} I\le {}^tA(x) A(x)\le C I$ for $\mu$-a.e.\,$x$ in the quadratic form sense.

Let $\psi\in L^{d/2}(\mu)+L^\infty(\mu)$. That is, let $\psi=\hat\psi+\check\psi$ for some $\hat\psi\in L^{d/2}(\mu)$ and $\check\psi\in L^\infty(\mu)$.
Take any $\eps>0$. By using \eqref{eq:sobolev}, for any $f \in \D$,
\begin{align*}
\int_{E}|\hat\psi|f^{2}\,d\mu &\le \| \bone_{\{ |\hat\psi|>N \} } \hat\psi\|_{d/2} \|f\|_{2d/(d-2)}^2+ N\int_{\{|\hat\psi|\le N \}} f^{2}\,d\mu \\
&\le S\| \bone_{\{|\hat\psi|>N \} }\hat\psi\|_{d/2}\cE^{0}_1(f)+ N^2\|f\|^{2}_{2}
\end{align*}
for any $N>0$. 
Take $N$ large enough that $S\| \bone_{\{|\hat\psi|>N \} }\hat\psi\|_{d/2}\le \eps$.
Then,
\[
\int_{E}|\psi|f^{2}\,d\mu 
\le \eps\cE^0(f)+(\eps+N^2+\|\check\psi\|_\infty)\|f\|_2^2.
\]
Therefore, $L^{d/2}(\mu)+L^\infty(\mu)\subset\cT_{0+}$.
In particular, due to Proposition~\ref{prop:suff}, each of 
Assumption~\ref{as:1}, (B.1), and (B.2$'$) with $\gm=0$ is satisfied 
if $|b|_H,|c|_H\in L^d(\mu)+L^\infty(\mu)$ and $V\in L^{d/2}(\mu)+L^\infty(\mu)$. 
In this case, the logarithmic term in (B.2$'$) is not useful.
\end{example}
\begin{example}\label{ex:2}
Assume that $\mu(E)<\infty$ and that $(\cE^{0},\D)$ satisfies the defective logarithmic Sobolev inequality:
\begin{equation} \label{eq:logsobolev}
\int_{E}f^{2} \log ( f^2/\|f\|_{2}^2 )\,d\mu \le \a\cE^{0}(f)+\b\|f\|_2^2,\quad f \in \D.
\end{equation}
A typical example is the following: $(E,H,\mu)$ is an abstract Wiener space,
$D$ is the $H$-derivative in the sense of the Malliavin calculus, 
and $\D$ is the first-order $L^2$ Sobolev space.
In this case, we can take $\a=4$ and $\b=0$.

Let $\psi\in L^0(\mu)$.
By the Hausdorff--Young inequality $st \le s\log s-s+e^t$ for $s\ge0$ and $t\in\R$, 
\begin{align*}
\dl |\psi|g^{2} \le g^2\log g^2-g^2+e^{\dl|\psi|}\quad
\text{for $\dl>0$ and $g\in L^2(\mu)$}.
\end{align*}
Taking $g=f/\|f\|_{2}$ for $f\in\D$ and using \eqref{eq:logsobolev}, we have
\begin{align*}
\int_{E}|\psi|f^{2}\,d\mu \le 
\frac{\a}{\dl} \cE^{0}(f)+\frac{ \b-1+\|e^{\dl|\psi|} \|_{1}}{\dl}\|f\|_{2}^2.
\end{align*}
Therefore, $|\psi|\in \cT_{\a/\dl}$ if $e^{\dl|\psi|}\in L^1(\mu)$.
Assumption~\ref{as:1}, (B.1), and (B.2$'$)  hold if 
\[
e^{\dl_1|b+c|^2_H}, e^{\dl_2|b-c|^2_H}, e^{\dl_3 V_+}, e^{\dl_4 V_-} \in L^1(\mu) \text{ with $\dl_i>0$ $(i=1,2,3,4)$ and $\sqrt{2\a/\dl_1}+\a/\dl_4<1$}, 
\]
by applying Proposition~\ref{prop:suff} with $b_1=0$, $b_2=b$, $c_1=0$, and $c_2=c$. 
We cannot expect that (B.2$'$) will hold with $\gm=0$ in general; thus, the introduction of the logarithmic term in (B.2$'$) is effective in this case.
\end{example}






\end{document}